\documentclass[12pt,reqno]{amsart}
\usepackage{enumerate}
\usepackage{amsrefs}
\usepackage{mathrsfs}
\usepackage{amsfonts, amsmath, amssymb, amscd, amsthm, bm, cancel}
\usepackage{url}
\usepackage{graphicx}
\usepackage[linktocpage=true,colorlinks,citecolor=magenta,linkcolor=blue,urlcolor=magenta]{hyperref}
\usepackage{multicol}
\usepackage{comment}
\usepackage[T1]{fontenc}
\usepackage[margin=1in]{geometry}
\usepackage{tikz}
\makeatletter
\@namedef{subjclassname@2020}{\textup{2020} Mathematics Subject Classification}
\makeatother
 \newtheorem{thm}{Theorem}[section]
 
 \newtheorem{prob}{Problem}[section]
 \newtheorem{cor}[thm]{Corollary}
 
 \newtheorem{lem}[thm]{Lemma}
 
 \newtheorem{prop}[thm]{Proposition}
 \theoremstyle{definition}
 \newtheorem{defn}[thm]{Definition}
 \newtheorem*{ack}{Acknowledgments}
 \theoremstyle{remark}
 \newtheorem{rem}[thm]{Remark}
 \newtheorem{ex}{Example}
 
 \numberwithin{equation}{section}
 \numberwithin{figure}{section}
 \allowdisplaybreaks

\renewcommand{\(}{\left(}
\renewcommand{\)}{\right)}

\begin{document}
\title[The discrete horospherical $p$-Minkowski problem in hyperbolic space]{The discrete horospherical $p$-Minkowski problem in hyperbolic space}

\author{Haizhong Li}
\address{Department of Mathematical Sciences, Tsinghua University, Beijing 100084, P.R. China}
\email{\href{mailto:lihz@tsinghua.edu.cn}{lihz@tsinghua.edu.cn}}

\author{Yao Wan}
\address{Department of Mathematical Sciences, Tsinghua University, Beijing 100084, P.R. China}
\email{\href{mailto:y-wan19@mails.tsinghua.edu.cn}{y-wan19@mails.tsinghua.edu.cn}}

\author{Botong Xu}
\address{Department of Mathematics, Technion--Israel Institute of Technology, Haifa 32000, Israel} 
\email{\href{mailto:botongxu@campus.technion.ac.il}{botongxu@campus.technion.ac.il}}
% \address{Department of Mathematical Sciences, Tsinghua University, Beijing 100084, P.R. China}
% \email{\href{mailto:xbt17@tsinghua.org.cn}{xbt17@tsinghua.org.cn}}


%%% ----------------------------------------------------------------------

\begin{abstract}
In \cite{LX}, the first author and the third author introduced and studied the horospherical $p$-Minkowski problem for smooth horospherically convex domains in hyperbolic space. In this paper, we introduce and solve the discrete horospherical $p$-Minkowski problem in hyperbolic space for all $p\in(-\infty,+\infty)$ when the given measure is even on the unit sphere.
\end{abstract}

\maketitle
%\tableofcontents

\section{Introduction}\label{sec:1}
A central problem in the Brunn-Minkowski theory in Euclidean space $\mathbb{R}^{n+1}$ is the Minkowski problem which asks if a given Borel measure $\mu$ on the unit sphere $\mathbb{S}^{n}$ arises as the surface area measure of a convex body. The existence of the solution to this problem for Borel measures was given by Alexandrov \cite{Ale42} and independently by Fenchel and Jessen \cite{FJ38}. The solution is unique up to translation, and its regularity was studied by Lewy \cite{lewy38}, Nirenberg \cite{Nir53}, Cheng-Yau \cite{CY76}, Pogorelov \cite{Pog78}, and Caffarelli \cite{Caf90}. When the given measure $\mu$ is discrete, the following discrete Minkowski problem can be understood as prescribing the surface areas of facets of a polytope, and the problem was solved by Minkowski \cites{Min1897, Min1903} himself.

\textbf{Discrete Minkowski problem in Euclidean space.}\ 
\textit{Let $\mu$ be a discrete measure on $\mathbb{S}^n$. Find necessary and sufficient conditions on $\mu$ so that there exists a convex polytope $P$ in $\mathbb{R}^{n+1}$ whose surface area measure with respect to the Gauss map is the given measure $\mu$. }

It is a natural question to ask how to propose the extension of the above prescribed discrete surface area measure problem in hyperbolic space. In this paper, we substitute the convexity in Euclidean space by the horospherical convexity (or h-convexity for short) in hyperbolic space, which is widely studied in hyperbolic geometry. The polytopes in Euclidean space are given by the intersections of finite half spaces. Similarly, we call a closed domain in the hyperbolic space $\mathbb{H}^{n+1}$ an \emph{h-convex polytope} if it is given by the intersection of finite closed horo-balls (see section \ref{sec:2} for more details). Now we ask the following problem, which can be understood as prescribing the surface areas of facets of an h-convex polytope.

\begin{prob}[\textbf{Prescribed discrete horospherical surface area measure problem in hyperbolic space}]\label{Problem 1.1}
Let $\mu$ be a discrete measure on $\mathbb{S}^n$. Find necessary and sufficient conditions on $\mu$ so that there exists a h-convex polytope $P$ in $\mathbb{H}^{n+1}$ whose horospherical surface area measure with respect to the horospherical Gauss map is a multiple of the given measure $\mu$.
\end{prob}

In the Poincar\'e ball model $(\mathbb{B}^{n+1}, g_B)$ of $\mathbb{H}^{n+1}$, the horospheres are spheres tangent to $\partial \mathbb{B}^{n+1}$, and the horo-balls are domains delimited by horospheres. Hence the following Figure \ref{Fig-1} gives a perception of the relationship between the discrete Minkowski problem in $\mathbb{R}^2$ and Problem \ref{Problem 1.1} in $\mathbb{H}^2$.

\begin{figure}[htbp]
\centering
\begin{tikzpicture}[scale=0.4][>=Stealth] 
\filldraw[opacity=0.6, draw=blue!70, fill=blue!100] (-2,0)--(1,3)--(1,-3)--cycle;

\fill (0,0) circle (2pt);
\node[right] at (-0.3,-1) {{\footnotesize{$P$}}};
\draw[-stealth] (1,0)--(1.5,0);
\node[above right] at (1.3,0) {{\footnotesize{$\Vec{e}_1$}}};
\draw[-stealth] (-1,1)--(-1.3,1.3);
\node[above] at (-1.3,1.3) {{\footnotesize{$\Vec{e}_2$}}};
\draw[-stealth] (-1,-1)--(-1.3,-1.3);
\node[left] at (-1.3,-1.3) {{\footnotesize{$\Vec{e}_3$}}};
\node[left] at (4,-3) {{\footnotesize{$\mathbb{R}^{2}$}}};
\draw[-] (-4,-2)--(2,4);
\draw[-] (-4,2)--(2,-4);
\draw[-] (1,-4.5)--(1,4.5);
\end{tikzpicture}
\hspace{2.5em}
\begin{tikzpicture}[scale=0.4][>=Stealth] 
\filldraw[opacity=0.6, draw=blue!70, fill=blue!100] (1.6,-1.24)arc(-1:84:2.5);
\filldraw[opacity=0.6, draw=blue!70, fill=blue!100] (0.15,-2.104)arc(-75.46:-42.84:3);
\filldraw[opacity=0.6, draw=blue!70, fill=blue!100]  (-0.643,1.287)arc(149:237.3:2.5);
\filldraw[opacity=0.6, draw=blue!70, fill=blue!100] (-0.643,1.287)--(0.15,-2.104)--(1.6,-1.24)--cycle;

\draw (0,0) circle (4);
\fill (0,0) circle (2pt);
\node[right] at (-0.3,-1) {{\footnotesize{$P$}}};
\draw[-stealth] (4,0)--(4.5,0);
\node[above right] at (4,0) {{\footnotesize{$\Vec{e}_1$}}};
\draw[-stealth] (-2.4,3.2)--(-2.7,3.6);
\node[above] at (-2.3,3.3) {{\footnotesize{$\Vec{e}_2$}}};
\draw[-stealth] (-2.4,-3.2)--(-2.7,-3.6);
\node[left] at (-2.5,-3.2) {{\footnotesize{$\Vec{e}_3$}}};
\node[left] at (6,-3) {{\footnotesize{$\mathbb{B}^{2}$}}};
\draw (1.5,0) circle (2.5);
\draw (-0.6,0.8) circle (3);
\draw (-0.9,-1.2) circle (2.5);
\end{tikzpicture}
\caption{Discrete Minkowski problem in $\mathbb{R}^{2}$ and Problem \ref{Problem 1.1} in $\mathbb{H}^{2}$}
 \label{Fig-1}
\end{figure}
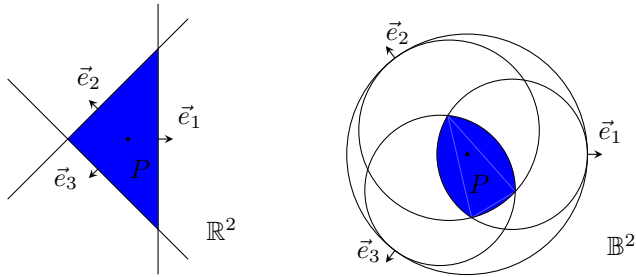

$\ $

The $L_p$ Minkowski problem in Euclidean space $\mathbb{R}^{n+1}$ was posted by Lutwak \cite{lut93}, where the case $p=1$ is the classical Minkowski problem. The $L_p$ Minkowski problem asks to characterize the $L_p$ surface area measure of a convex body, and it contains the logarithmic Minkowski problem ($p=0$) and the centro-affine Minkowski problem ($p=-n-1$) as special cases. The $L_p$ Minkowski problem has been extensively studied in recent decades, see e.g. \cites{lut93, lut04, Chen06, Lu13, CW06, BLY13, BT17, JLZ16}; see also a recent survey by B\"{o}r\"{o}czky \cites{Bor22}.

The $L_p$ Minkowski problem  for polytopes are of great importance. One reason is that the $L_p$ Minkowski problem ($p\ge 1$) for Borel measures can be solved by an approximation argument by first solving the polytopal case, see e.g. \cites{HLYZ, sch14}. Note that if a polytope $P$ contains the origin in its interior with $N$ facets whose outward unit normals are $\{\textbf{e}_1,\ldots,\textbf{e}_N\}\subset\mathbb{S}^n$, and if the facet with outward unit normal $\textbf{e}_i$ has area $a_i$ and distance from the origin $u(P,\textbf{e}_i)$ for $i=1,\ldots,N$. Then the $L_p$ surface area measure of $P$ is defined by
$$S_p(P,\cdot)=\sum\limits_{i=1}^N u(P,\textbf{e}_i)^{1-p}a_i\delta_{\textbf{e}_i}(\cdot),$$
where $\delta_{\textbf{e}_i}$ denotes the delta measure concentrated at the point $\textbf{e}_i$ on $\mathbb{S}^n$. For a given discrete measure $\mu$, the $L_p$ Minkowski problem can be stated in the following way:

\textbf{Discrete $L_p$ Minkowski problem in Euclidean space.}\ 
\textit{Let $\mu$ be a discrete measure on $\mathbb{S}^n$. Find necessary and sufficient conditions on $\mu$ so that there exists a convex polytope $P$ in $\mathbb{R}^{n+1}$ whose $L_p$ surface area measure $S_p(P,\cdot)$ is the given measure $\mu$.}

The discrete $L_p$ Minkowski problem in Euclidean space was treated by, e.g., Hug, Lutwak, Yang and Zhang \cite{HLYZ} for $p>1$, B\"{o}r\"{o}czky, Heged\H{u}s and Zhu \cite{Bor16} for $p=0$, and Zhu \cites{zhu1501, zhu1502, zhu14, zhu17} for $p<1$.

In \cite{LX}, the first author and the third author introduced the horospherical $p$-surface area measures of smooth uniformly h-convex bounded domains in hyperbolic space $\mathbb{H}^{n+1}$ and have solved the corresponding horospherical $p$-Minkowski problem when the given measure is even on $\mathbb{S}^n$. Here a Borel measure $\mu$ on $\mathbb{S}^n$ is said to be \textit{even} if $\mu(\omega)=\mu(-\omega)$ for any Borel subset $\omega\subset\mathbb{S}^n$. In \cite{LW}, the first author and the second author investigated the Christoffel problem in the hyperbolic plane (i.e. the horospherical $p$-Minkowski problem in the case $n=1$ and $p=-1$) and presented existence results without the evenness assumption on $\mu$. In the smooth category, the horospherical $p$-Minkowski problem in hyperbolic space can be stated in the following way:

\textbf{Horospherical $p$-Minkowski problem in hyperbolic space.}\ 
\textit{Let $\mu$ be a finite Borel measure on $\mathbb{S}^n$. Find necessary and sufficient conditions on $\mu$ so that there exists a h-convex body $K$ in $\mathbb{H}^{n+1}$ whose horospherical $p$-surface area measure $S_p(K,\cdot)$ with respect to the horospherical Gauss map is a multiple of the given measure $\mu$.}

The hyperboloid model of the hyperbolic space $\mathbb{H}^{n+1}$ in the Minkowski space $\mathbb{R}^{n+1,1}$ is given by $\mathbb{H}^{n+1}=\{X=(\textbf{x},x_{n+2})\in\mathbb{R}^{n+1,1}:\ X\cdot X=-1,\ x_{n+2}>0,\ \textbf{x}\in\mathbb{R}^{n+1}\}$. If a h-convex polytope $P$ in $\mathbb{H}^{n+1}$ contains the origin $O=(\textbf{0},1)$ in its interior and has $N$ facets with horospherical normals $\{\textbf{e}_1,\ldots,\textbf{e}_N\}\subset\mathbb{S}^n$, where each facet with horospherical normal $\textbf{e}_i$ has area $a_i$ and horospherical support function $u(P,\textbf{e}_i)$, then we define the horospherical $p$-surface area measure of $P$ by
$$S_p(P,\cdot)=\sum\limits_{i=1}^N e^{-pu(P,\textbf{e}_i)}a_i\delta_{\textbf{e}_i}(\cdot).$$
For a given discrete measure $\mu$, the horospherical $p$-Minkowski problem can be stated as follows:

\begin{prob}[\textbf{Discrete horospherical $p$-Minkowski problem in hyperbolic space}]\label{Problem 1.2}
Let $\mu$ be a discrete measure on $\mathbb{S}^n$. Find necessary and sufficient conditions on $\mu$ so that there exists a h-convex polytope $P$ in $\mathbb{H}^{n+1}$ whose horospherical $p$-surface area measure $S_p(P,\cdot)$ is a multiple of the given measure $\mu$.
\end{prob}

When $p=0$, Problem \ref{Problem 1.2} is reduced to Problem \ref{Problem 1.1}. When $p=-n$, we call it the \textit{discrete horospherical Minkowski problem}.

In this paper, we consider Problem \ref{Problem 1.2} in the case that $\mu$ is even on $\mathbb{S}^{n}$. Our main result is the following Theorem \ref{Thm1.1}.

\begin{thm}\label{Thm1.1}
Let $p\in\mathbb{R}$ and $\mu$ be an even discrete measure on $\mathbb{S}^n$. Then there exists an origin-symmetric h-convex polytope $P$ in $\mathbb{H}^{n+1}$, such that the measure $\mu$ is a multiple of its horospherical $p$-surface area measure $S_p(P,\cdot)$.
\end{thm}

The approach to prove Theorem \ref{Thm1.1} can be outlined as follows. First, we prove a selection theorem for h-convex polytopes. Additionally, we derive a variational formula for the volume functional of a h-convex polytope. Next, we divide Problem \ref{Problem 1.2} into two cases: $p \geq 0$ and $p<0$, and construct constrained optimization problems for each case such that their critical points are the desired solutions. Finally, by using the Lagrange multiplier method, the selection theorem and the variational formula proved in the first step, we prove the existence of solutions for the above two optimization problems.

The paper is organized as follows. 
In section \ref{sec:2}, we provide some basic definitions and properties about h-convex bodies in hyperbolic space.
In section \ref{sec:3}, we study the Hausdorff metric on h-convex bodies and prove a selection theorem for h-convex polytopes. In section \ref{sec:4}, we study horospherical Wulff shape and derive a variational formula of volume functional of h-convex polytopes. In section \ref{sec:5}, we prove a boundedness lemma for h-convex bodies. In section \ref{sec:6}, we give the proof of Theorem \ref{Thm1.1} by studying two constrained optimization problems.

\begin{ack}
	The work was supported by NSFC Grant No. 11831005 and NSFC Grant No. 12126405.
\end{ack}

\section{Preliminaries}\label{sec:2}

In this section, we collect some definitions and properties of horospherically convex bodies in hyperbolic space. We refer to \cites{And21,esp09,LX} for details.

Consider the Minkowski space $\mathbb{R}^{n+1,1}$ with canonical coordinates $(x_1,\ldots,x_{n+2})$ and the Lorentzian metric
$$X\cdot Y=\sum\limits_{i=1}^{n+1}x_i y_i-x_{n+2}y_{n+2}.$$
The hyperbolic space is then realized in the Minkowski space as the hyperboloid
$$\mathbb{H}^{n+1}=\{X=(x_1,\ldots,x_{n+2})\in\mathbb{R}^{n+1,1}:\ X\cdot X=-1,\ x_{n+2}>0\}.$$
Denote by $\hat{B}_r(X)$ the geodesic ball of radius $r$ centered at $X$ in $\mathbb{H}^{n+1}$, and denote by $d(X,Y)$ the geodesic distance between $X$ and $Y$ on $\mathbb{H}^{n+1}$. We call $O=(\textbf{0},1)$ the origin of the hyperbolic space $\mathbb{H}^{n+1}$.

An important class of totally umbilical hypersurfaces in $\mathbb{H}^{n+1}$ is the so-called \textit{horosphere}. Specifically, the horospheres in $\mathbb{H}^{n+1}$ are the hypersurfaces with principal curvatures equal to 1 everywhere.

Now we begin to parameterize the set of horospheres. Let $\Sigma$ be a horosphere in $\mathbb{H}^{n+1}$, denote by $\nu$ the outward unit normal of $\Sigma\subset\mathbb{H}^{n+1}$, and let $c=X-\nu$ be a nonzero vector-valued function defined on $\Sigma$. For any tangent vector $v \in T_X \mathbb{H}^{n+1}$, we then have 
$$\nabla_v c=\nabla_v(X-\nu)=0,$$
which implies that $c$ is constant along $\Sigma$. Moreover, by
$$X\cdot c=X\cdot(X-\nu)=-1,$$
we can describe the horosphere $\Sigma$ as the intersection of the null hyperplane $\{X\in \mathbb{R}^{n+1,1}:\ X\cdot c=-1\}$ and the hyperboloid $\mathbb{H}^{n+1}$. Since $c\cdot c=0$, we can assume $c=\lambda(\textbf{e},1)$ for some $\textbf{e}\in\mathbb{S}^n$ and $\lambda \in \mathbb{R} \backslash \{0\}$, then
$$-1=X\cdot c= \lambda X \cdot (\textbf{e},1).$$
On the other hand, since
$$X\cdot ( \textbf{e},1)=(x_1,\ldots, x_{n+1}) \cdot \textbf{e} - x_{n+2} \le \left(
\sum\limits_{i=1}^{n+1}x_i^2
\right)^{\frac{1}{2}}-x_{n+2} <0,$$
we have $\lambda>0$, and hence we can define $s=-\log \lambda$. Consequently, the set of horospheres in $\mathbb{H}^{n+1}$ can be parameterized by $(\textbf{e}, s) \in \mathbb{S}^n \times \mathbb{R}$.

Denote by $H_\textbf{e}(s)$ the horosphere
\begin{equation}\label{2.1}
    H_\textbf{e}(s)=\{X\in\mathbb{H}^{n+1}:\ X\cdot(\textbf{e},1)=-e^s\},
\end{equation}
where $\textbf{e}$ is called the center of $H_\textbf{e}(s)$.
Denote by $B_\textbf{e}(s)$ the \textit{horo-ball} enclosed by $H_\textbf{e}(s)$, i.e.
\begin{equation}\label{2.2}
    B_\textbf{e}(s)=\{X\in\mathbb{H}^{n+1}:\ 0>X\cdot(\textbf{e},1)>-e^s\}.
\end{equation}

Now we show that the parameter $s$ in $H_\textbf{e}(s)$ represents the signed geodesic distance from $O$ to $H_\textbf{e}(s)$. In fact, for any $X \in H_\textbf{e}(s)$, there exist $\theta \in \mathbb{S}^n$ and $r \geq 0$ such that $X=(\sinh r \theta,\cosh r)$. Here $r$ represents the geodesic distance between $X$ and $O$. Then
$$-e^r\le -e^s=X\cdot (\textbf{e},1)=\sinh r \theta\cdot \textbf{e}-\cosh r\le -e^{-r},$$
thus the geodesic distance between $H_\textbf{e}(s)$ and $O$ is $|s|$. The sign of $s$ determines whether the origin $O$ is inside $B_\textbf{e}(s)$.

In addition to the hyperboloid model $\mathbb{H}^{n+1}$, we will also use the Poincar\'e ball model $(\mathbb{B}^{n+1}, g_B)$ and the upper half-space model $(\mathbb{U}^{n+1}, g_U)$ of the hyperbolic space in the context.

The stereographic projection $\pi$ with respect to $(\textbf{0},-1)$ maps the hyperboloid model $\mathbb{H}^{n+1}$ to the Poincar\'e ball model $\mathbb{B}^{n+1}$, which is given by
\begin{equation*}
    \pi (x_1, \ldots, x_{n+1}, x_{n+2}) = \frac{(x_1, \ldots, x_{n+1})}{1+x_{n+2}}.
\end{equation*}
Conversely, for $Y = (y_1, \ldots, y_{n+1}) \in \mathbb{B}^{n+1}$, we have
\begin{equation*}
    \pi^{-1}(Y) = \left(\frac{2Y}{1- |Y|^2}, \frac{1+|Y|^2}{1- |Y|^2} \right).
\end{equation*}
This together with \eqref{2.1} and \eqref{2.2} then gives
\begin{align*}
    \pi \left(H_{\textbf{e}}(s) \right)=& \left\{ Y \in \mathbb{B}^{n+1} : \ \left|  Y - \frac{\textbf{e}}{1+e^s}\right| = \frac{e^s}{1+e^s}  \right\},\\
     \pi \left(B_{\textbf{e}}(s) \right)=& \left\{ Y \in \mathbb{B}^{n+1} : \ \left|  Y - \frac{\textbf{e}}{1+e^s}\right| < \frac{e^s}{1+e^s}  \right\}.
\end{align*}
Thus, in the Poincar\'e ball model $\mathbb{B}^{n+1}$, the horosphere $H_{\textbf{e}} (s)$ is a sphere tangent to $\textbf{e} \in \partial \mathbb{B}^{n+1}$. So the horospheres can be viewed as `spheres' centered at the infinity $\partial \mathbb{B}^{n+1}$ of the hyperbolic space.

The upper half-space model $(\mathbb{U}^{n+1}, g_{U})$ of the hyperbolic space is given by
\begin{equation*}
   \mathbb{U}^{n+1}=\{Y=(\textbf{y},y_{n+1}):\ \textbf{y}=(y_1,\ldots,y_n)\in\mathbb{R}^n,\ y_{n+1}>0\},\quad
g_U=\frac{1}{y_{n+1}^2}\sum\limits_{i=1}^{n+1} dy_i^2. 
\end{equation*}
Note that $ \partial \mathbb{U}^{n+1} = (\mathbb{R}^{n} \times \{0\} )\cup \{ \infty\}$ and $\mathbb{S}^n$ can be viewed as identical through an isometry between $(\mathbb{U}^{n+1},g_U)$ and $(\mathbb{B}^{n+1},g_B)$. If $\textbf{e} =\infty\in \partial \mathbb{U}^{n+1}$, then it follows from $g_U$ that
 \begin{equation*}
    H_\textbf{e}(s) = \{Y\in\mathbb{U}^{n+1}: y_{n+1}=e^{-s}\},\quad 
    B_\textbf{e}(s) = \{Y\in\mathbb{U}^{n+1}: y_{n+1} >e^{-s}\}.
\end{equation*}
If $\textbf{e} \in \mathbb{R}^{n} \times \{0\}$, then the expression for  $H_{\textbf{e}}(s)$ can be obtained by the Euclidean inversion of $H_{\infty} (s)$ with respect to the hemisphere centered at $\textbf{e}$ passing through $(\textbf{0},1)$, see e.g. Lopez's book \cite[p.193]{Lop}.

\begin{defn}
A closed domain $\Omega$ in $\mathbb{H}^{n+1}$ is called \textit{horospherically convex }(or \textit{h-convex} for short) if each boundary point $X$ of $\partial\Omega$ has a supporting horo-ball, i.e. a horo-ball $B$ such that $\Omega\subset\overline{B}$ and $X\in\partial B$. 
\end{defn}

The above definition is equivalent to the following condition on $\Omega$: for each pair of points in $\Omega$, all the entire horocycle arcs joining them are also contained in $\Omega$, see e.g. \cite{GST13}.

\begin{defn}
A \textit{h-convex body} is a compact h-convex subset of $\mathbb{H}^{n+1}$ with nonempty interior. Denote by $\mathcal{K}^h$ the class of h-convex bodies.
\end{defn}
Let us introduce notations for two subsets of $\mathcal{K}^h$.
\begin{itemize}
    \item Denote by $\mathcal{K}_o^h$ the class of h-convex bodies containing the origin $O$ in their interiors.
    \item Denote by $\mathcal{K}_e^h$ the class of origin-symmetric h-convex bodies.
\end{itemize}

Let $K\in\mathcal{K}^h$ be a h-convex body, and let $\textbf{e}\in \mathbb{S}^n$. Define the \textit{horospherical support function} of $K$ in direction $\textbf{e}$ by 
\begin{equation}\label{2.3}
    u(K,\textbf{e})=\inf \{s\in\mathbb{R}:\ K\subset \overline{B}_\textbf{e}(s)\}.
\end{equation}
 Then by (\ref{2.2}),
\begin{equation}\label{2.4}
    u(K,\textbf{e})=\sup \{f_\textbf{e}(X):\ X\in K\},
\end{equation}
where $f_\textbf{e}(X)=\log(-X\cdot(\textbf{e},1))$.
Moreover, $B_\textbf{e}(u(K,\textbf{e}))$ is called the \textit{supporting horo-ball} of $K$ in direction $\textbf{e}$.

When $\partial K$ is smooth, we define its \textit{horospherical Gauss map} $G:\partial K\to \mathbb{S}^n$ by setting $G(X)=\textbf{e}$ for $X\in\partial K$, where $\textbf{e}$ satisfies $X-\nu=\lambda(\textbf{e},1)$ for some $\lambda>0$. When $\partial K$ is not smooth, we call $\textbf{e}\in \mathbb{S}^n$ its \textit{horospherical normal} at $X \in \partial K$ if $- (X-X') \cdot  (\textbf{e},1) \geq 0$ for all $X'\in K$. Equivalently,  $X\in H_\textbf{e}(u(K,\textbf{e}))\cap \partial K$ and $X' \in B_\textbf{e}(u(K,\textbf{e}))$ for all $X' \in K$.

The main objects studied in this paper are the following h-convex polytopes.
\begin{defn}
A h-convex body $P$ is called a \textit{h-convex polytope} if it can be represented as the intersection of finite closed horo-balls, i.e., there exist $\{\textbf{e}_1, \ldots, \textbf{e}_m\} \subset \mathbb{S}^n$ ($m \geq 2$) and $\{u_1, \ldots, u_m\} \subset \mathbb{R}$ such that 
\begin{equation}\label{2.5}
 P=\bigcap\limits_{i=1}^m\overline{B}_{\textbf{e}_i}(u_i).
\end{equation}
Denote by $\mathcal{P}^h$ the set of h-convex polytopes.
If $F(P,\textbf{e}_i):=P\cap H_{\textbf{e}_i}(u_i)$ is $n$-dimensional, then $F(P,\textbf{e}_i)$ is called the \textit{facet of $P$ with horospherical normal $\textbf{e}_i$}.
\end{defn}

For convenience, let us introduce some notations for various subsets of $\mathcal{P}^h$. Let $\{\textbf{e}_1,\ldots,\textbf{e}_m\}$ be unit vectors in $\mathbb{S}^n\ (m\ge 2)$.
\begin{itemize}
\item Denote by $\mathcal{P}_o^h$ the class of h-convex polytopes containing the origin $O$ in their interiors, i.e. $\mathcal{P}_o^h = \mathcal{P}^h \cap \mathcal{K}_o^h$.

\item Denote by $\mathcal{P}_e^h$ the class of origin-symmetric h-convex polytopes, i.e. $\mathcal{P}_e^h = \mathcal{P}^h \cap \mathcal{K}_e^h$.

\item Denote by $\mathcal{P}^h(\textbf{e}_1,\ldots,\textbf{e}_m)$ the subset of $\mathcal{P}^h$ such that a h-convex polytope $P \in \mathcal{P}^h(\textbf{e}_1,\ldots,\textbf{e}_m)$ if $P$ can be represented as
\begin{equation*}
    P=\bigcap\limits_{i=1}^m \overline{B}_{\textbf{e}_i}(u(P,\textbf{e}_i)).
\end{equation*}
Obviously, if $P\in\mathcal{P}^h(\textbf{e}_1,\ldots,\textbf{e}_m)$, then $P$ has at most $m$ facets, and its horospherical normals are in $\{\textbf{e}_1,\ldots,\textbf{e}_m\}$.

\item Denote by $\mathcal{P}_m^h(\textbf{e}_1,\ldots,\textbf{e}_m)$ the subset of $\mathcal{P}^h(\textbf{e}_1,\ldots,\textbf{e}_m)$ such that a h-convex polytope $P\in\mathcal{P}_m^h(\textbf{e}_1,\ldots,\textbf{e}_m)$ if $P$ has exactly $m$ facets. 
\end{itemize}

\section{Hausdorff metric on h-convex bodies and a selection theorem}\label{sec:3}

For any $K \in \mathcal{K}^h$, denote by  $K^\epsilon$ the outer parallel set of $K$ with distance $\epsilon$, i.e.
\begin{equation*}
K^\epsilon=\{X\in\mathbb{H}^{n+1}:\ d(K,X) \leq \epsilon\}.
\end{equation*}
The following Proposition \ref{Prop3.1} gives the relationship between the horospherical support function of $K^{\epsilon}$ and that of $K$.

\begin{prop}\label{Prop3.1}
	Let $K\in\mathcal{K}^h$ and $\epsilon>0$. Then the horospherical support function of $K^\epsilon$ in direction $\textbf{e}$ is given by
	\begin{equation}\label{3.1}
	u(K^{\epsilon},\textbf{e})=u(K,\textbf{e})+\epsilon.
	\end{equation}
\end{prop}
\begin{proof}
	We work in the upper half-space model $(\mathbb{U}^{n+1},g_U)$. Without loss of generality, we can assume that
	\begin{equation*}
	H_\textbf{e}(s) = \{Y\in\mathbb{U}^{n+1}: y_{n+1}=e^{-s}\}.
	\end{equation*}
	Taking $s= u(K,\textbf{e})$ and $s= u(K, e)+\epsilon$ respectively, we have
	\begin{align*}
	H_\textbf{e}(u(K,\textbf{e})) =& \{Y\in\mathbb{U}^{n+1}: y_{n+1}=e^{-u(K, \textbf{e})}\},\\
	H_\textbf{e}(u(K,\textbf{e})+\epsilon) =& \{Y\in\mathbb{U}^{n+1}: y_{n+1}=e^{-u(K, \textbf{e})-\epsilon}\}.
	\end{align*}
	Clearly, the geodesic distance between $H_\textbf{e}(u(K,\textbf{e}))$ and $H_\textbf{e}(u(K,\textbf{e})+\epsilon)$ is $\epsilon$. Thus $$B_\textbf{e}(u(K,\textbf{e}))^{\epsilon}=B_\textbf{e}(u(K,\textbf{e})+\epsilon).$$
	Note that the fact $K\subset B_\textbf{e}(u(K,\textbf{e}))$ induces $K^{\epsilon} \subset B_\textbf{e}(u(K,\textbf{e}))^{\epsilon}$.
	Then we have
	\begin{equation}\label{3.2}
	u(K^{\epsilon},\textbf{e}) \leq u(K,\textbf{e})+\epsilon.
	\end{equation}

	On the other hand, the definition of $u(K, \textbf{e})$ in (\ref{2.3}) implies $\partial K \cap  H_{\textbf{e}}(u(K,\textbf{e})) \neq \emptyset$. Then there exists $\bar{Y} =( \bar{y}_1, \ldots, \bar{y}_{n+1} ) \in\partial K \cap  H_{\textbf{e}}(u(K,\textbf{e}))$. Let $Y^* = ( y^*_1, \ldots y^*_{n+1} )$, where $y^*_i = \bar{y}_i$, $i=1,\ldots, n$, and $y^*_{n+1} = e^{-\epsilon}\bar{y}_{n+1}$.
	Then $Y^*\in H_\textbf{e}(u(K,\textbf{e})+\epsilon)$ and $d(\bar{Y},Y^*)=\epsilon$, which implies $Y^* \in K^{\epsilon}$.
	Consequently, 
	\begin{equation}\label{3.3}
	u(K^{\epsilon},\textbf{e}) \geq u(K,\textbf{e})+\epsilon.
	\end{equation}
	Then the desired formula (\ref{3.1}) follows from (\ref{3.2}) and (\ref{3.3}). We complete the proof of Proposition \ref{Prop3.1}.
\end{proof}

Recall that the Hausdorff metric between two sets $K, L\in \mathcal{K}^h$ is defined by
\begin{equation}\label{3.4}
d_{\mathcal{H}}(K,L)=\min\{\epsilon\ge 0:\ K\subset L^\epsilon,\ L\subset K^{\epsilon}\}.
\end{equation}

\begin{cor}\label{Cor3.2}
	Let $K, L \in \mathcal{K}^h$. Then
	\begin{equation}\label{3.5}
	d_{\mathcal{H}}(K,L)=\max\{|u(K,\textbf{e})-u(L,\textbf{e})|:\ \textbf{e}\in \mathbb{S}^n\}.
	\end{equation}
\end{cor}
\begin{proof}
	By (\ref{3.1}), the inclusions $K \subset L^{\epsilon}$ and $L \subset K^{\epsilon}$ are equivalent to, for all $\textbf{e} \in \mathbb{S}^n$,
	\begin{equation*}
	u(K, \textbf{e}) \leq u(L, \textbf{e}) +\epsilon \quad {\rm and} \quad 
	u(L, \textbf{e}) \leq u(K, \textbf{e}) +\epsilon,
	\end{equation*}
	respectively. Then by (\ref{3.4}), we have
	\begin{align*}
	d_{\mathcal{H}}(K,L)=&\min\{\epsilon\ge 0:   | u(K,\textbf{e})-u(L,\textbf{e}) |\leq \epsilon \ {\rm for \ all} \ \textbf{e} \in \mathbb{S}^n  \}\\
	=&\max\{|u(K,\textbf{e})-u(L,\textbf{e})|:\ \textbf{e}\in \mathbb{S}^n\}.
	\end{align*}
	This completes the proof of Corollary \ref{Cor3.2}.
\end{proof}

\begin{defn}
	We say that a sequence of h-convex bodies $\{K_i\}_{i=1}^{\infty}$ \textit{converges to} a compact h-convex subset $K\subset\mathbb{H}^{n+1}$ \textit{with respect to the Hausdorff metric} if
	$$d_{\mathcal{H}}(K_i,K)\to 0,\quad \text{as}\ i\to\infty.$$
\end{defn}

For any $K\in\mathcal{K}_o^h$, the radial function of $K$ in direction $\theta\in\mathbb{S}^n$ is defined by
$$\rho(K,\theta)=\max\{\lambda\ge0: (\sinh \lambda\cdot\theta,\cosh\lambda)\in K\}.$$
Now we proceed to show that the convergence of h-convex bodies with respect to the Hausdorff metric on $\mathcal{K}_o^h$ is equivalent to the uniform convergence of their radial functions on $\mathbb{S}^n$. See \cite{sch14} for the Euclidean case. Denote
\begin{align}\label{3.6}
R(K) :=\max\limits_{\theta\in\mathbb{S}^n}\rho(K,\theta),\quad
r(K) :=\min\limits_{\theta\in\mathbb{S}^n}\rho(K,\theta).
\end{align}
It follows from (\ref{2.4}) that
\begin{align*}
u(K,\textbf{e})
=\sup\limits_{\theta\in \mathbb{S}^n}\log(\cosh \rho(K,\theta)-\sinh \rho(K,\theta)\theta\cdot\textbf{e})
\ge \rho(K,-\textbf{e}).
\end{align*}
Then we get
\begin{equation*}
\max\limits_{\textbf{e}\in\mathbb{S}^n} u(K,\textbf{e})\ge R(K), \quad \min\limits_{\textbf{e}\in\mathbb{S}^n} u(K,\textbf{e})\ge r(K).
\end{equation*}
Assume that $\textbf{e}_0\in\mathbb{S}^n$ attains the maximum of $u(K,\textbf{e})$, we have
\begin{align*}
\max\limits_{\textbf{e}\in\mathbb{S}^n} u(K,\textbf{e})
=u(K,\textbf{e}_0)
\le\sup\limits_{\theta\in \mathbb{S}^n}\log(\cosh \rho(K,\theta)+\sinh \rho(K,\theta))=R(K).
\end{align*}

On the other hand, suppose that $\theta_0\in\mathbb{S}^n$ attains the minimum of $\rho(K,\theta)$. Let us consider the Poincar\'e ball model $(\mathbb{B}^{n+1},g_B)$, see Figure \ref{Fig-3-1}. Denote by $P$ the boundary point of $K$ that satisfies $d(P,O)=r(K)$, and denote by $B^*=B_{\textbf{e}^*}(s^*)$ the supporting horo-ball of $K$ at $P$. Since $K$ contains the geodesic ball $\hat{B}=\hat{B}_{r(K)}(O)$, and there exists a unique horo-ball $B_{-\theta_0}(r(K))$ such that $\hat{B}$ is internally tangent to it at $P$, we then derive that $\textbf{e}^*=-\theta_0$ and $s^*=r(K)$. Thus
\begin{align*}
r(K)=u(K,-\theta_0)\ge \min\limits_{\textbf{e}\in\mathbb{S}^n} u(K,\textbf{e}).
\end{align*}
Putting the above facts together, we obtain
\begin{equation}\label{3.7}
\max\limits_{\textbf{e}\in\mathbb{S}^n} u(K,\textbf{e})=R(K),\quad
\min\limits_{\textbf{e}\in\mathbb{S}^n} u(K,\textbf{e})= r(K).
\end{equation}
This together with (\ref{3.1}) gives
\begin{equation}\label{3.8}
R(K^{\epsilon}) = R(K)+\epsilon, \quad r(K^\epsilon) =r(K)+\epsilon.
\end{equation}

\begin{figure}[htbp]
	\centering
	\begin{tikzpicture}[scale=0.5][>=Stealth] 
	\draw (0,0) circle (5);
	\draw (0,0) circle (2);
	\fill (0,0) circle (3pt);
	\node[below] at (0,0) {{\footnotesize{$O$}}};
	\node[below] at (1.2,-1.7) {{\footnotesize{$\hat{B}$}}};
	\node[below] at (0,-3.3) {{\footnotesize{$B^*$}}};
	
	\draw (1.5,0) circle (3.5);
	\fill (-2,0) circle (3pt);
	\node[right] at (-2,0) {{\footnotesize{$P$}}};
	
	\draw[-stealth] (5,0)--(5.5,0);
	\node[above right] at (5,0) {{\footnotesize{$\textbf{e}^*$}}};
	\draw[-stealth] (-2,0)--(-2.5,0);
	\node[above left] at (-2,0) {{\footnotesize{$\theta_0$}}};
	
	\node[left] at (8,-3) {$\mathbb{B}^{n+1}$};
	\end{tikzpicture}
	\caption{\ }
	\label{Fig-3-1}
\end{figure}

In the following lemma, we investigate the correlation between the radial functions of $K$ and $K^\epsilon$.
\begin{lem}\label{Lem3.4}
	Let $K\in\mathcal{K}_o^h$. Then for any $0<\epsilon<1$ and any $\theta \in \mathbb{S}^n$, we have
	\begin{equation}\label{3.9}
	\rho(K^\epsilon,\theta)\le\rho(K,\theta)+\frac{2e^{2R(K)}}{r(K) e^{r(K)}} \epsilon.
	\end{equation}
\end{lem}

\begin{proof}
	Let $\theta$ be any fixed point on $\mathbb{S}^n$.
	Denote $\rho(K,\theta)$ as $\rho$ and $\rho(K^\epsilon,\theta)$ as $\rho_\epsilon$. Let $X=(\sinh\rho\theta,\cosh\rho)$ and $X^\epsilon=(\sinh\rho_\epsilon \theta,\cosh\rho_\epsilon)$ be the boundary points  of $K$ and $K^\epsilon$ in direction $\theta$, respectively. Obviously, $\rho \leq R(K)$ and $\rho_\epsilon > \rho>0$.
	
	Since $X\in\partial K$, it follows from the h-convexity of $K$ that there exists $\textbf{e}\in\mathbb{S}^n$ satisfying
	\begin{align}\label{3.10}
	u(K,\textbf{e})=\log(-X\cdot(\textbf{e},1))=\log(\cosh\rho-\sinh \rho\theta\cdot \textbf{e}).
	\end{align}
	This together with (\ref{3.1}) shows
	\begin{equation}\label{3.11}
	u(K^\epsilon,\textbf{e}) 
	=u(K,\textbf{e})+\epsilon
	=\log(\cosh\rho-\sinh \rho \theta\cdot \textbf{e})+\epsilon.
	\end{equation}   
	Besides, we know from (\ref{2.4}) and $X^{\epsilon} \in \partial K^{\epsilon} $ that 
	\begin{equation}\label{3.12}
	u(K^\epsilon,\textbf{e})  
	\ge \log(-X^\epsilon\cdot(\textbf{e},1))
	=\log(\cosh\rho_\epsilon-\sinh \rho_\epsilon \theta\cdot \textbf{e}).
	\end{equation}
	Comparing (\ref{3.11}) with (\ref{3.12}) gives
	\begin{align}\label{e-epsl-1}
	e^\epsilon-1 
	&\ge \frac{\cosh\rho_\epsilon-\sinh \rho_\epsilon \theta\cdot \textbf{e}}{\cosh\rho-\sinh \rho \theta\cdot \textbf{e}}-1 \nonumber\\
	&= \frac{ \(\cosh\rho_\epsilon- \cosh \rho\) -(\sinh \rho_\epsilon-\sinh \rho)\theta\cdot \textbf{e}}{ \cosh\rho-\sinh \rho \theta\cdot\textbf{e}} \nonumber\\
	&=2\sinh\frac{\rho_\epsilon-\rho}{2}\cdot\frac{\sinh\frac{\rho_\epsilon+\rho}{2}-\cosh\frac{\rho_\epsilon+\rho}{2} \theta\cdot \textbf{e}}{\cosh\rho-\sinh \rho \theta\cdot \textbf{e}} \nonumber\\
	&\ge 2 e^{-R(K)} \sinh\frac{\rho_\epsilon-\rho}{2}\left(\sinh\frac{\rho_\epsilon+\rho}{2}-\cosh\frac{\rho_\epsilon+\rho}{2} \theta\cdot \textbf{e}\right), 
	\end{align}
	where we used $\rho \leq R(K)$ in the last inequality. Note that (\ref{3.10}) implies
	\begin{equation}\label{3.14}
	\theta\cdot \textbf{e}=\frac{\cosh \rho- \exp{(u(K,\textbf{e}))}}{\sinh\rho}.
	\end{equation}
	Substituting (\ref{3.14}) into the right-hand side of \eqref{e-epsl-1} yields
	\begin{align*}
	e^\epsilon-1 
	&\ge 2 e^{-R(K)}
	\sinh\frac{\rho_\epsilon-\rho}{2}\cdot
	\left(\sinh\frac{\rho_\epsilon+\rho}{2}
	-\cosh\frac{\rho_\epsilon+\rho}{2} \cdot
	\frac{\cosh \rho -\exp u(K, \textbf{e})  }{\sinh \rho }\right)\\
	&= 2 e^{-R(K)} \sinh\frac{\rho_\epsilon-\rho}{2}\cdot\frac{\exp{u(K,\textbf{e})}\cosh\frac{\rho_\epsilon+\rho}{2}-\cosh\frac{\rho_\epsilon-\rho}{2}}{\sinh\rho}\\
	&\ge 2 e^{-R(K)} \sinh\frac{\rho_\epsilon-\rho}{2}\cdot\frac{(\exp{u(K,\textbf{e})}-1)\cosh\frac{\rho_\epsilon+\rho}{2}}{\sinh\rho}\\
	&\ge 2 e^{-R(K)} \sinh\frac{\rho_\epsilon-\rho}{2}\cdot\frac{(e^{r(K)}-1)\cosh r(K)}{\sinh R(K)}\\
	&\ge \frac{r(K) e^{r(K)}}{e^{2R(K)}} (\rho_{\epsilon}-\rho),
	\end{align*}
	where we used (\ref{3.7}) and the facts that $2 \sinh \frac{\rho_{\epsilon} -\rho}{2} \geq \rho_\epsilon - \rho$, $e^{r(K)}-1 \geq r(K)$, $\sinh R(K) \leq \frac{1}{2}e^{R(K)}$ and $\cosh r(K) \geq \frac{1}{2} e^{r(K)}$. It is easy to see that $e^{\epsilon}-1 \leq 2\epsilon$ when $0<\epsilon<1$.
	Thus, for $0<\epsilon<1$ we have
	\begin{equation*}
	\rho_\epsilon-\rho\le \frac{e^{2R(K)}}{r(K) e^{r(K)}}(e^\epsilon-1)
	\le \frac{2e^{2R(K)}}{r(K) e^{r(K)}}\epsilon,
	\end{equation*}
	which is the desired inequality (\ref{3.9}). We complete the proof of Lemma \ref{Lem3.4}.
\end{proof}

\begin{thm}\label{Thm3.5}
Let $\{K_i\}_{i=1}^{\infty} \subset \mathcal{K}_o^h$ be a sequence of h-convex bodies, and let $K \in \mathcal{K}_o^h$. The following statements are equivalent,
\begin{enumerate}[(i)]
	\item $K_i$ converges to $K$ in Hausdorff metric as $i \to \infty$,
	\item $u(K_i,\cdot)$ converges uniformly on $\mathbb{S}^n$ to $u(K,\cdot)$ as $i \to \infty$,
	\item $\rho(K_i,\cdot)$ converges uniformly on $\mathbb{S}^n$ to $\rho(K,\cdot)$  as $i \to \infty$.
\end{enumerate}
\end{thm}

\begin{proof}
	The equivalence between the statements (i) and (ii) follows from (\ref{3.5}). It suffices to show that statements (i) and (iii) are equivalent.
	
	We first prove that statement (i) implies statement (iii). Let $\{K_i\}_{i=1}^{\infty} \subset \mathcal{K}_o^h$ be a sequence of h-convex bodies that converges to $K \in \mathcal{K}_o^h$ in Hausdorff metric as $i \to \infty$. For any given $0<\epsilon<\min\left\{1,\frac{r(K)}{2}\right\}$, we have  $d_{\mathcal{H}}(K_i,K)<\epsilon$, i.e. $K_i\subset K^\epsilon$ and $K\subset K_i^\epsilon$ for large $i$. This together with (\ref{3.8}) gives
	\begin{equation}\label{3.15}
	\begin{aligned}
	&R(K_i)\leq R(K^\epsilon) =  R(K)+\epsilon\le R(K)+1, \\\ 
	&r(K_i)= r(K_i^\epsilon) -\epsilon \ge r(K)-\epsilon\ge \frac{r(K)}{2}.
	\end{aligned}
	\end{equation}
	Define the function $C: \mathbb{R}^+ \times \mathbb{R} \to \mathbb{R}$ by
	\begin{equation*}
	C(r,R)=\frac{2e^{2R}}{r e^r}.
	\end{equation*}
	Clearly, the function $C(r, R)$ is monotone increasing in $R$ and decreasing in $r$. Using (\ref{3.9}) and (\ref{3.15}), for any $\theta \in \mathbb{S}^n$ we have
	\begin{align*}
	\rho(K_i,\theta) \leq \rho(K^\epsilon, \theta)\le \rho(K,\theta)+C(r(K),R(K))\epsilon,
	\end{align*}
	and
	\begin{align*}
	\rho(K,\theta)
	\leq \rho (K_i^\epsilon, \theta) \leq
	\rho(K_i,\theta)+C\left(r(K_i),R(K_i)\right)\epsilon
	\le\rho(K_i,\theta)+C\left(\frac{r(K)}{2},R(K)+1\right)\epsilon.
	\end{align*}
	Taking $\epsilon \to 0^+$ in the above two inequalities, we have that $|\rho(K_i,\cdot)-\rho(K,\cdot)|$ converges uniformly to $0$ as $i \to \infty$, which is the statement (iii). Thus we obtain that statement (i) implies statement (iii).
	
	Next, we prove that statement (iii) implies statement (i). Let $\{K_i\}_{i=1}^\infty \subset \mathcal{K}_o^h$ be a sequence of h-convex bodies such that $\rho(K_i,\cdot)$ converges uniformly to $\rho(K,\cdot)$ as $i \to \infty$. Given any $\epsilon>0$, we have  $|\rho(K_i,\theta)-\rho(K,\theta)|<\epsilon$ for each $\theta\in\mathbb{S}^n$ by choosing $i$ large enough. Denote by $X=(\sinh\rho(K,\theta)\theta,\cosh\rho(K,\theta))$ and $X_i=(\sinh\rho(K_i,\theta)\theta,\cosh\rho(K_i,\theta))$ the boundary points in direction $\theta$ of $K$ and $K_i$, respectively. 
	Then
	\begin{align*}
	d(X,X_i) =|\rho(K_i,\theta)-\rho(K,\theta)|<\epsilon.
	\end{align*}
	Hence we have $\partial K_i \subset K^{\epsilon}$ and $\partial K \subset K_i^\epsilon$, which induces  $d_{\mathcal{H}}(K_i,K)<\epsilon$ by (\ref{3.4}). Therefore, $K_i$ converges to $K$ in Hausdorff metric as $i \to \infty$, which is the statement (i). Thus we obtain that statement (iii) implies statement (i).
	
	We complete the proof of Theorem \ref{Thm3.5}.
\end{proof}

According to the facts that the volume of $K\in\mathcal{K}_o^h$ can be expressed as
\begin{equation*}
V(K)=\int_{\mathbb{S}^n}\int_0^{\rho(K,\theta)}\sinh^n r dr d\theta
\end{equation*}
and that the volume functional is invariant under isometries, Theorem \ref{Thm3.5} implies:

\begin{cor}\label{Cor3.6}
	The volume functional $V(\cdot)$ in hyperbolic space is continuous for h-convex bodies under the Hausdorff metric.
\end{cor}

\begin{defn}
	We say that a sequence of h-convex bodies $\{K_i\}_{i=1}^{\infty}$ is \textit{bounded} if they are all contained in a geodesic ball of fixed radius, i.e., there exists a constant $M>0$ such that
	$$\max\limits_{X\in K_i}d(O,X)\le M$$
	holds for all $i \geq 1$. 
\end{defn}

In Euclidean space, the Blaschke selection theorem says that every bounded sequence of compact convex sets in $\mathbb{R}^{n+1}$ has a subsequence that converges to a compact convex set.
Now we prove a selection theorem in hyperbolic space.

\begin{thm}\label{Thm3.8}
	Let $\{K_i\}_{i=1}^{\infty}$ be a bounded sequence of h-convex bodies in $\mathbb{H}^{n+1}$. Then there exists a subsequence $\{K_{i_k}\}$ such that $u(K_{i_k},\cdot)$ converges uniformly to a continuous function. 
	
	Moreover, suppose that $\{K_i\}_{i=1}^{\infty}$ is a bounded sequence of h-convex polytopes in $\mathcal{P}_o^h(\textbf{e}_1,\ldots,\textbf{e}_m)$ given by $K_i=\bigcap\limits_{j=1}^m\overline{B}_{\textbf{e}_j}(x_{i,j})$, where $\{x_{i,j}\}$ has a positive uniform lower bound for all $i,j$. Then there exists a subsequence $\{K_{i_k}\}$ that converges to a h-convex polytope $K \in \mathcal{P}_o^h(\textbf{e}_1,\ldots,\textbf{e}_m)$ in Hausdorff metric.
	
\end{thm}
\begin{proof}
	Since the sequence $\{K_i\}_{i=1}^{\infty}$ is bounded, there exists a constant $M>0$ such that $\max\limits_{X\in K_i} d(O,X)\le M$ for all $i \geq 1$.
	
	First, we show that $\{u(K_i,\cdot) \}_{i=1}^{\infty}$ is uniformly bounded. Fix some $i \geq 1$.
	Note that each $X\in K_i $ can be represented as $X=(\sinh r_X \theta_X,\cosh r_X)$, where $r_X=d(O,X)\le M$ and $\theta_X\in\mathbb{S}^n$. By (\ref{2.4}), for any $\textbf{e}\in\mathbb{S}^n$ we have
	\begin{align*}
	u(K_i,\textbf{e}) 
	=\sup\limits_{X\in K_i}\log(-X\cdot(\textbf{e},1)) 
	=\sup\limits_{X\in K_i}\log(\cosh r_X-\sinh r_X \theta_X\cdot \textbf{e}).
	\end{align*}
	Then
	\begin{align*}
	u(K_i,\textbf{e})
	\le \sup\limits_{X\in K_i}\log(\cosh r_X+\sinh r_X)
	=\sup\limits_{X\in K_i} r_X\le M,
	\end{align*}
	and
	\begin{align*}
	u(K_i,\textbf{e})
	\ge \sup\limits_{X\in K_i}\log(\cosh r_X-\sinh r_X)
	=-\inf\limits_{X\in K_i} r_X\ge -M.
	\end{align*}
	This implies that $\{u(K_i,\cdot) \}_{i=1}^{\infty}$ is uniformly bounded on $\mathbb{S}^n$.

	Next, we show that $\{u(K_i,\cdot)\}_{i=1}^{\infty}$ is equicontinuous. For any integer $i\ge1$ and
 $\textbf{e}_1\in \mathbb{S}^n$,
 %unit vectors $\textbf{e}_1,\textbf{e}_2\in \mathbb{S}^n$, 
 there exists $X_i\in K_i$ such that
	\begin{equation}\label{3.16}
	u(K_i,\textbf{e}_1)=\log(-X_i\cdot(\textbf{e}_1,1)).
	\end{equation}
	Denote $X_i=(\sinh r_i\theta_i,\cosh r_i)$, where $ r_i=d(O,X_i)\le M$ and $ \theta_i\in \mathbb{S}^n$. 
 %By (\ref{2.4}), we also have
 For any $\textbf{e}_2\in \mathbb{S}^n$, we have by (\ref{2.4})
	\begin{equation}\label{3.17}
	u(K_i,\textbf{e}_2)\ge\log(-X_i\cdot(\textbf{e}_2,1)).
	\end{equation}
	It follows from (\ref{3.16}) and (\ref{3.17}) that
	\begin{align*}
	u(K_i,\textbf{e}_1)-u(K_i,\textbf{e}_2)
	\le& \log(-X_i\cdot(\textbf{e}_1,1))-\log(-X_i\cdot(\textbf{e}_2,1))\\
	=&\log\frac{\cosh r_i-\sinh r_i \theta_i\cdot \textbf{e}_1}{\cosh r_i-\sinh r_i \theta_i\cdot \textbf{e}_2}\\
	\le& \max\left\{0,\ \frac{\sinh r_i \theta_i\cdot(\textbf{e}_2-\textbf{e}_1)}{\cosh r_i-\sinh r_i\theta_i\cdot \textbf{e}_2}\right\}\\
	\le& \max\left\{0,\ \frac{e^{2r_i}-1}{2}|\textbf{e}_2-\textbf{e}_1|\right\}\\
	\le& \frac{e^{2M}-1}{2}|\textbf{e}_1-\textbf{e}_2|,
	\end{align*}
	where we used $\log(1+t) \leq t$ for $t \geq 0$ in the second inequality. Thus we obtain
	\begin{equation*}
	|u(K_i,\textbf{e}_1)-u(K_i,\textbf{e}_2)|\le\frac{e^{2M}-1}{2}|\textbf{e}_1-\textbf{e}_2|,\quad i=1,2,\ldots,
	\end{equation*}
	for all $\textbf{e}_1,\textbf{e}_2\in \mathbb{S}^n$. Consequently, $\{u(K_i,\cdot)\}_{i=1}^{\infty}$ is equicontinuous on $\mathbb{S}^n$. 
	
	It follows from the Arzel\`a-Ascoli theorem that there exists a subsequence $u(K_{i_k},\cdot)$ that converges uniformly on $\mathbb{S}^n$ to a continuous function $u_0(\cdot)$. This completes the proof of the first statement of Theorem \ref{Thm3.8}.
	
	%%%%%%%%%%%%%%%%%%%%%%%%%%%%%%%%%%%
	
	In addition, if $K_i$ has a discrete structure $\bigcap\limits_{j=1}^m \overline{B}_{\textbf{e}_j}(x_{i,j})$ which can be determined by a vector $(x_{i,j})_{1\le j\le m}$ in $\mathbb{R}^m$.
 Without loss of generality, we may assume $x_{i,j} \leq u(K_i, \textbf{e}_j)$.
 Then the above argument implies that there exists a subsequence $\{K_{i_k}\}$ such that the vector $(x_{i_k, j})$ converges to $(x_j)$ as $k\to\infty$. Then $K=\bigcap\limits_{j=1}^m\overline{B}_{\textbf{e}_j}(x_{j})$ is a h-convex polytope in $\mathcal{P}_o^h(\textbf{e}_1,\ldots,\textbf{e}_m)$ by assumption.
	
	For sufficiently large $k$, there exists a geodesic ball $\hat{B}_R(O)$ such that it contains all $K_{i_k}$ and $K$. Notice that $K_{i_k}=\bigcap\limits_{j=1}^m \left(\overline{B}_{\textbf{e}_j}(x_{i_k,j})\cap \hat{B}_R(O)\right)$, then its radial function is given by
	$$\rho(K_{i_k},\theta)=\min\limits_{1\le j\le m}\rho(\overline{B}_{\textbf{e}_j}(x_{i_k,j})\cap \hat{B}_R(O),\theta),\quad\forall\ \theta\in\mathbb{S}^n,$$
	where $\rho(\overline{B}_{\textbf{e}_j}(x_{i_k,j})\cap \hat{B}_R(O),\theta)=\min\{\rho(\overline{B}_{\textbf{e}_j}(x_{i_k,j}),\theta),R\}$ is a continuous bounded function. Since $x_{i_k,j}$ converges to $x_j$ as $k\to\infty$, we have that $\rho(\overline{B}_{\textbf{e}_j}(x_{i_k,j})\cap \hat{B}_R(O),\cdot)$ converges uniformly on $\mathbb{S}^n$ to $\rho(\overline{B}_{\textbf{e}_j}(x_{j})\cap \hat{B}_R(O),\cdot)$ 
 for all $j=1,\ldots,m$. Moreover, $\rho(K_{i_k},\cdot)$ converges uniformly on $\mathbb{S}^n$ to $\rho(K,\cdot)$. Finally, by applying Theorem \ref{Thm3.5}, we conclude that $K_{i_k}$ converges to $K$ in Hausdorff metric. We complete the proof of Theorem \ref{Thm3.8}.
\end{proof}

% \begin{rem}\label{Rem2.15}
%  We conjecture that the limiting function $u_0$ can be viewed as a  horospherical support function of a compact h-convex set $K_0\subset\mathbb{H}^{n+1}$.
% \end{rem}

%$\ $

\begin{defn}
	The \textit{h-convex hull} of a set $A$ in $\mathbb{H}^{n+1}$ is the intersection of all h-convex sets containing $A$.
\end{defn}

It is clear that the h-convex hull of $A$ contains all horocycle arcs joining points in $A$, and the h-convex hull of finite points in $\mathbb{H}^2$ is a h-convex polytope. However, the h-convex hull of finite points in $\mathbb{H}^{n+1}$ is not a h-convex polytope when $n\ge2$, since there exist exactly two horocycle arcs joining two points in $\mathbb{H}^{2}$, and there exist infinitely many horocycle arcs joining two points in $\mathbb{H}^{n+1}$ for $n\ge 2$.

Using a similar argument as the case in Euclidean space, we note that a h-convex body in $\mathbb{H}^2$ can be approximated by a sequence of h-convex polytopes.

\begin{thm}\label{Thm3.10}
	Let $K$ be a h-convex body in $\mathbb{H}^{2}$. Then there exists a sequence of h-convex polytopes ${P_i}$ such that $d_{\mathcal{H}}(K,P_i)\to 0$ as $i\to\infty$.
\end{thm}
\begin{proof}
	For any $\epsilon>0$, there exist $X_1,\ldots,X_N\in K$ such that the geodesic balls $\{\hat{B}_{\epsilon}(X_i)\}$ cover $K$. Denote by $P_{\epsilon}$ the h-convex hull of $\{X_i\}_{i=1}^N$, and by $(P_{\epsilon})^{\epsilon}$ the outer parallel set of $P_{\epsilon}$ with distance $\epsilon$. Since $(P_{\epsilon})^{\epsilon}$ is a h-convex body containing all $\hat{B}_{\epsilon}(X_i)$, then $(P_{\epsilon})^{\epsilon}$ contains $K$. Therefore, the h-convex polytope $P_{\epsilon}$ satisfies $P_{\epsilon}\subset K\subset (P_{\epsilon})^{\epsilon}$ and $d_{\mathcal{H}}(K,P_{\epsilon})\le \epsilon$. Then Theorem \ref{Thm3.10} follows by taking $\epsilon \to 0^+$.
\end{proof}

\section{Horospherical Wulff shape and a variational formula of volume}\label{sec:4}

\begin{defn}\label{def-Wulff shape}
	Let $\omega$ be a closed subset of $\mathbb{S}^n$ that contains at least two elements, and let $g$ be a positive continuous function on $\mathbb{S}^n$. We define \textit{the horospherical Wulff shape associated with $(g,\omega)$} as
	\begin{align}\label{4.1}
	K_g=\bigcap\limits_{\textbf{e}\in \omega}\overline{B}_\textbf{e}(g(\textbf{e})).
	\end{align} 
	By (\ref{2.2}), we have
	\begin{equation*}
	\begin{split}
	K_g &=\bigcap\limits_{\textbf{e}\in \omega} \{ X\in\mathbb{H}^{n+1}:0> X\cdot(\textbf{e},1)\ge-\exp{g(\textbf{e})}\}\\
	&=\bigcap\limits_{\textbf{e}\in \omega} \{X\in\mathbb{H}^{n+1}:\ f_\textbf{e}(X)\le g(\textbf{e})\},
	\end{split}
	\end{equation*}
	where $f_\textbf{e}(X)=\log(-X\cdot(\textbf{e},1)).$
\end{defn}

It is worth noting that any h-convex body $K$ can be fully characterized by its horospherical support function $u(K,\textbf{e})$, i.e.
\begin{equation}\label{4.2}
K=\bigcap\limits_{\textbf{e}\in \mathbb{S}^n}\overline{B}_\textbf{e}(u(K,\textbf{e})).
\end{equation}
To prove this, it suffices to prove the following separating horosphere theorem in hyperbolic space. Denote by $d(\cdot, \cdot)$ the distance function in $\mathbb{H}^{n+1}$.

\begin{thm}\label{Thm4.2}
Let $K$ be a h-convex body in $\mathbb{H}^{n+1}$ and 
let $Q$ be a point outside $K$. Then there exists a horosphere $H_{\textbf{e}}(s)$ separating $K$ and $Q$, i.e. $K\subset\overline{B}_{\textbf{e}}(s)$ and $Q\notin \overline{B}_{\textbf{e}}(s)$.
\end{thm}
\begin{proof}
	We work in the Poincar\'e ball model $(\mathbb{B}^{n+1}, g_B)$, see Figure \ref{Fig-4-1}. Up to an isometry, we may assume that $Q$ is the center of $\mathbb{B}^{n+1}$. It follows from the h-convexity of $K$ that there exists a unique point $P\in \partial K$ such that $d(P,Q)=d(K,Q)$. Let $\hat{B}$ be the geodesic ball of radius $d(Q,P)$ centered at $Q$, and let $B$ be the horo-ball tangent to $\hat{B}$ at $P$. We will show that the horosphere $\partial B$ separates $K$ and $Q$, i.e. $K\subset\overline{B}$ and $Q\notin\overline{B}$.
	
	We argue by contradiction. Suppose that there exists a point $P_0\in K\backslash\overline{B}$. Since $B$ is the unique horo-ball externally tangent to $\hat{B}$ at $P$, the intersection of $\hat{B}$ and the horo-ball containing $P$ and $P_0$ on the boundary has interior points.
 Thus there exists a point $P^*$ such that $P^*\in \hat{B}$, and $P^*$ lies on a horocycle arc connecting $P$ and $P_0$, thus $d(P^*, Q)< d(P,Q)$. By the h-convexity of $K$, we have $P^*\in K$. However, the definition of $P$ implies $d(P,Q) \leq d(P^*, Q)$, which is a contradiction. This completes the proof of Theorem \ref{Thm4.2}.

\end{proof}
\begin{figure}[htbp]
	\centering
	\begin{tikzpicture}[scale=0.5][>=Stealth] 
	\draw (0,0) circle (5);
	\draw (0,0) circle (2);
	\fill (0,0) circle (3pt);
	\node[below] at (0,0) {{\footnotesize{$Q$}}};
	\node[below] at (0,-2) {{\footnotesize{$\hat{B}$}}};
	
	\draw (3.5,0) circle (1.5);
	\fill (2,0) circle (3pt);
	\node[right] at (2,0) {{\footnotesize{$P$}}};
	\node[below] at (3.5,-1.5) {{\footnotesize{$B$}}};
	
	\draw (3.25,0.9986) circle (1.6);
	\fill (2,1.9975) circle (3pt);
	\node[above] at (2,1.9975) {{\footnotesize{$P_0$}}};
	
	\fill (1.7,0.602) circle (3pt);
	\node[above right] at (1.7,0.8) {{\footnotesize{$P^*$}}};
	
	\node[left] at (8,-3) {$\mathbb{B}^{n+1}$};
	\end{tikzpicture}
	\caption{Separating horosphere theorem in $\mathbb{H}^{n+1}$}
	\label{Fig-4-1}
\end{figure}
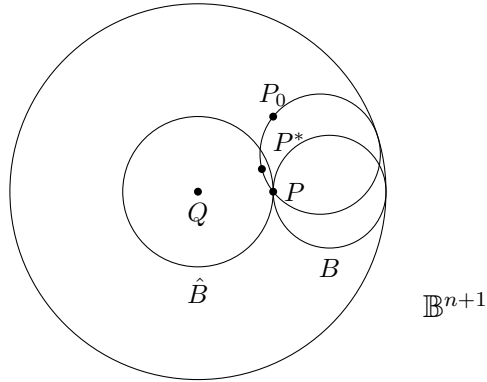

Back to Definition \ref{def-Wulff shape}. It is easy to see that $K_g$ is a h-convex body that contains $O=(\textbf{0},1)$ in its interior. Note that
\begin{equation}\label{4.3}
u_{K_g}(\textbf{e})\le g(\textbf{e}),\quad \forall\ \textbf{e}\in\omega.
\end{equation}

In particular, if $\omega\subset \mathbb{S}^n$ is a finite set $\{\textbf{e}_1,\ldots,\textbf{e}_m\}\ (m\ge2)$ and $x\in\mathbb{R}_{+}^m=\{(x_1,\ldots,x_m)\in\mathbb{R}^m:\ x_i>0,\ 1\le i\le m\}$, %\textcolor{blue}{NOTE:ALLOW $x_i=0$}
then the \textit{horospherical Wulff shape associated with $(\omega,x)$} is a h-convex polytope defined as
\begin{equation}\label{4.4}
P(\omega,x)
=\bigcap\limits_{i=1}^m\overline{B}_{\textbf{e}_i}(x_i).
\end{equation}

Conversely, any $P \in \mathcal{P}_o^h$ can be viewed as a horospherical Wulff shape associated with a discrete pair $(\omega,x)$.

$\ $

In the rest part of this section, we will derive a variational formula for h-convex polytopes in Lemma \ref{Lem4.6} below. 
The following local Steiner formula proved by Kohlmann \cite{koh91} will be used to prove Lemma \ref{Lem4.6}.

\begin{lem}[\cite{koh91}]\label{Lem4.3}
	Let $A$ be a geodesically convex set in $\mathbb{H}^{n+1}$. The map $f_A:\mathbb{H}^{n+1}\backslash A\to \partial A$ is defined by $ d(f_A(x),x)=d(K,x)$. For a bounded Borel set $B\subset \mathbb{H}^{n+1}$ and $\epsilon>0$, define
	\begin{equation*}
	\mathscr{P}_{\epsilon}(A,B)
	=f_A^{-1}(B)\cap (A^{\epsilon}\backslash A)
	=\{x\in \mathbb{H}^{n+1}|\ 0<d(A,x)\le\epsilon,\ f_A(x)\in B\}.
	\end{equation*}
	Then we have 
	\begin{equation}\label{4.5}
	V(\mathscr{P}_{\epsilon}(A,B))
	=\sum\limits_{k=0}^n l_{n+1-k}(\epsilon)\Phi_k(A,B),
	\end{equation}
	where
	\begin{align*}
	& l_{n+1-k}(t)
	=\int_0^t \cosh^k(\tau) \sinh^{n-k}(\tau) d\tau,
	\quad k=0,\ldots,n,
	\end{align*}
	and $\Phi_k(A,\cdot)$ is the $k$-th curvature measure of $A$ on Borel sets. 
\end{lem}

Furthermore, 
%as in the classical setting of convex bodies in the Euclidean space,
the curvature measures introduced by Kohlmann are weakly continuous with respect to the Hausdorff metric on $\mathcal{K}$, where $\mathcal{K}$ is the class of compact geodesically convex sets in $\mathbb{H}^{n+1}$ with nonempty interior.

\begin{lem}[\cite{Ver19}]\label{Lem4.4}
	Let $\{A_j\}\subset\mathcal{K}$ be a sequence of geodesically convex sets such that $A_j\to A$ as $j\to\infty$ in the Hausdorff metric. Then for every $k=0,\ldots,n$ we have
	\begin{align*}
	\Phi_k(A_j,\cdot)\to \Phi_k(A,\cdot)
	\end{align*}
	as $j\to\infty$, weakly in the sense of measure.
\end{lem}

Let us turn to the case of h-convex polytopes. Suppose that $P=\bigcap\limits_{i=1}^m\overline{B}_{\textbf{e}_i}(x_i)$ and $u(P,\textbf{e}_i)=x_i$ for $1\le i\le m$. Denote the facet $P\cap H_{\textbf{e}_i}(x_i)$ as $F(P,\textbf{e}_i)$ and its interior as $\mathring{F}(P,\textbf{e}_i)$. Let $P_t$ be the set $\mathscr{P}_t(P,\mathring{F}(P,\textbf{e}_i))$ as defined in Lemma \ref{Lem4.3} for any $t>0$, and let $S(P,\textbf{e}_i)$ be the surface area of $F(P,\textbf{e}_i)$. Now we calculate the expression for $V(P_t)$ specifically in the upper half-space model $(\mathbb{U}^{n+1},g_U)$. Without loss of generality, we may assume that $\mathring{F}(P,\textbf{e}_i)$ is given by
\begin{equation*}
\mathring{F}(P,\textbf{e}_i)= \{ (\textbf{y}, y_{n+1}) \in \mathbb{U}^{n+1}: \ \textbf{y} \in \hat{F}(P,\textbf{e}_i), \ y_{n+1} =e^{-x_i} \}, 
\end{equation*}
where $\hat{F}(P,\textbf{e}_i)$ is a domain in $\mathbb{R}^{n}$. Let $\hat{S}(P,\textbf{e}_i)$ be the volume of $\hat{F}(P,\textbf{e}_i)$ in $\mathbb{R}^{n}$. Then we have $S(P,\textbf{e}_i)=e^{n x_i}\hat{S}(P,\textbf{e}_i)$ and
\begin{align}\label{4.6}
P_t=\{(\textbf{y},\bar{y})\in\mathbb{U}^{n+1}:\ \textbf{y}\in \hat{F}(P,\textbf{e}_i),\ e^{-x_i-t}\le \bar{y}< e^{-x_i}\}.
\end{align}
Since the volume element of $(\mathbb{U}^{n+1},g_U)$ is $(y_{n+1})^{-(n+1)}dy_{n+1}d\textbf{y}$, we obtain
\begin{equation}\label{4.7}
\begin{split}
V(P_t)=\int_{e^{-x_i-t}}^{e^{-x_i}}\int_{P_t\cap\{y_{n+1}=\bar{y}\}}\frac{d\bar{y} d\textbf{y}}{\bar{y}^{n+1}}
=\frac{1}{n}(e^{n x_i+n t}-e^{n x_i})\hat{S}(P,\textbf{e}_i)
=\frac{e^{nt}-1}{n}S(P,\textbf{e}_i).
\end{split}
\end{equation}

The following lemma is also needed in the proof of Lemma \ref{Lem4.6}.

\begin{lem}\label{Lem4.5}
	Let $\textbf{e}_1,\ \textbf{e}_2\in\mathbb{S}^{n}$ and $x_1,\ x_2>0$. Let $P(x)=\overline{B}_{\textbf{e}_1}(x_1)\cap\overline{B}_{\textbf{e}_2}(x_2)$ and $P_t$ be the set $\mathscr{P}_t(P(x),\mathring{F}(P(x),\textbf{e}_1))$ as defined in Lemma \ref{Lem4.3} for any $t>0$. Define
	\begin{align*}
	M(t)=\left(\overline{B}_{\textbf{e}_1}(x_1+t)\backslash\overline{B}_{\textbf{e}_1}(x_1)\right)\bigcap \overline{B}_{\textbf{e}_2}(x_2).
	\end{align*}
	Then 
	\begin{align*}
	V(M(t)\Delta P_t)\le C t^2,\quad \text{as}\ t\to 0^+,
	\end{align*}
	where $\Delta$ denotes the symmetric difference between two sets, and $C$ is a positive constant depending only on $\{\textbf{e}_1,\textbf{e}_2\}$ and $(x_1,x_2)$.
\end{lem}

\begin{proof}
	We will estimate the volume of $M(t)\Delta P_t$ in the upper half-space model $(\mathbb{U}^{n+1},g_U)$. 
 
    At first, we assume that $B_{\textbf{e}_1}(s) = \{ Y \in \mathbb{U}^{n+1}:\ y_{n+1} > e^{-s} \}$ and $B_{\textbf{e}_2}(s) = \{ Y \in \mathbb{U}^{n+1}:\ 0\le |\textbf{y}|\le \sqrt{e^s y_{n+1}-y_{n+1}^2},\ 0<y_{n+1}\le e^{s} \}$, see Figure \ref{Fig-4-2}.
    Define $R(y)=\sqrt{e^{x_2}y-y^2}$. 
	Then $P_t$ is given by
	\begin{align*}
	P_t=\{(\textbf{y},y_{n+1})\in\mathbb{U}^{n+1}:\ 0\le |\textbf{y}|< R(e^{-x_1}),\ e^{-x_1-t}\le y_{n+1}< e^{-x_1}\},
	\end{align*}
	and $M(t)$ is given by
	\begin{align*}
	M(t)=\{(\textbf{y},y_{n+1})\in\mathbb{U}^{n+1}:\ 0\le |\textbf{y}|\le R(y_{n+1}),\ e^{-x_1-t}\le y_{n+1}< e^{-x_1}\}.
	\end{align*}
	Thus 
	\begin{align*}
	V(M(t)\Delta P_t)
	&=\int_{e^{-x_1-t}}^{e^{-x_1}}\int_{\left(M(t)\Delta P(t)\right)\cap\{y_{n+1}=\bar{y}\}}\frac{d\bar{y} d\textbf{y}}{\bar{y}^{n+1}}
	=\omega_n\int_{e^{-x_1-t}}^{e^{-x_1}}\frac{|R^n(e^{-x_1})-R^n(\bar{y})|}{\bar{y}^{n+1}}d\bar{y},
	\end{align*}
	where $\omega_n$ is the volume of the unit ball in $\mathbb{R}^n$. By the mean value theorem, there exists $\xi \in [e^{-x_1-t}, e^{-x_1}]$ such that
	\begin{align}\label{4.8}
	V(M(t)\Delta P_t) 
	&=\omega_n\left| R^n(e^{-x_1})-R^n(\xi)\right|\int_{e^{-x_1-t}}^{e^{-x_1}}\frac{d\bar{y}}{\bar{y}^{n+1}} \nonumber\\
	&=\frac{\omega_n}{n}\left| R^n(e^{-x_1})-R^n(\xi)\right| e^{n x_1}(e^{nt}-1).
	\end{align}
	Using the mean value theorem again, we have
	\begin{equation}\label{4.9}
	\frac{R^n(e^{-x_1})-R^n(\xi)}{x_1-(-\log\xi)}
	=-\frac{n}{2}\left( e^{x_2-\eta}-2 e^{-2\eta}\right)\left(e^{x_2-\eta}-e^{-2\eta}\right)^{\frac{n}{2}-1},
	\end{equation}
	where $x_1\le \eta\le -\log\xi\le x_1+t$. Notice that $0\le -\log \xi -x_1\le t$. Substituting (\ref{4.9}) into (\ref{4.8}), we obtain
	\begin{align*}
	V(M(t)\Delta P_t)
	&=\frac{\omega_n}{2}e^{n x_1}\left| e^{x_2-\eta}-2 e^{-2\eta}\right|\left(e^{x_2-\eta}-e^{-2\eta}\right)^{\frac{n}{2}-1}\cdot (-\log\xi-x_1)(e^{nt}-1)\\
	&\le C t^2,\quad \text{as}\ t\to 0^+,
	\end{align*}
	where $C$ is a positive constant depending only on $x_1$ and $x_2$.
 
    In general, there exists an isometry $T$ in hyperbolic space such that $B_{\Tilde{\textbf{e}}_1}(\Tilde{x}_1)=T(B_{\textbf{e}_1}(x_1))$ and $B_{\tilde{\textbf{e}}_2}(\tilde{x}_2)=T(B_{\textbf{e}_2}(x_2))$, where $B_{\Tilde{\textbf{e}}_1}(\Tilde{x}_1) = \{ Y:\ y_{n+1} > e^{-\Tilde{x}_1} \}$ and $B_{\Tilde{\textbf{e}}_2}(\tilde{x}_2) = \{ Y:\ 0\le |\textbf{y}|\le \sqrt{e^{\tilde{x}_2} y_{n+1}-y_{n+1}^2},\ 0<y_{n+1}\le e^{\tilde{x}_2} \}$. Then we have $T(P(x))=\overline{B}_{\Tilde{\textbf{e}}_1}(\Tilde{x}_1)\cap\overline{B}_{\Tilde{\textbf{e}}_2}(\Tilde{x}_2)$, where $\Tilde{x}_1$ and $\Tilde{x}_2$ are determined by $\{\textbf{e}_1,\textbf{e}_2\}$ and $(x_1,x_2)$. Consequently, by applying the above estimate, we can obtain
    \begin{align*}
	V(M(t)\Delta P_t)=V(T(M(t))\Delta T(P_t))\le \Tilde{C} t^2,\quad \text{as}\ t\to 0^+,
	\end{align*}
	where $\Tilde{C}$ is a positive constant depending only on $\{\textbf{e}_1,\textbf{e}_2\}$ and $(x_1,x_2)$. Moreover, $\Tilde{C}$ is continuous on $(x_1,x_2)$. This completes the proof of Lemma \ref{Lem4.5}.
\end{proof}

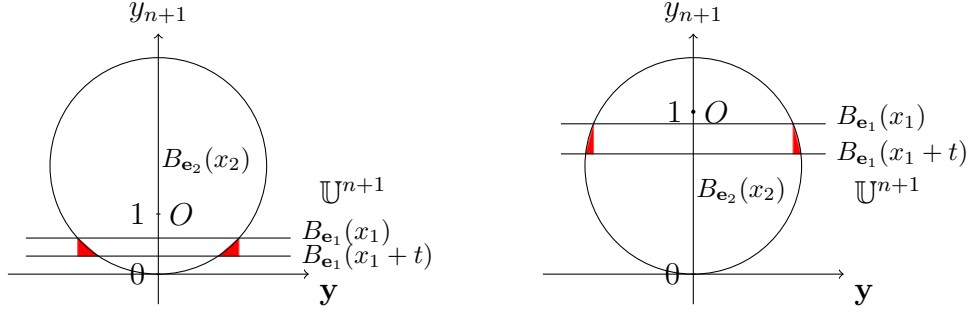
\begin{figure}[htbp]
	\centering
	\begin{tikzpicture}[scale=0.8][>=Stealth] 
	\draw[->](-2.5,0)--(2.5,0) node[below right]{$\textbf{y}$};
	\draw[->](0,-0.5)--(0,4) node[above]{$y_{n+1}$};
	\foreach \y in {0, 1} \draw (1pt, \y) -- (-1pt, \y) node[anchor=east] {$\y$};
	\filldraw[opacity=0.6, draw=red!70, fill=red!100] (1.34,0.6)--(1,0.3)--(1.34,0.3)--cycle;
	\filldraw[opacity=0.6, draw=red!70, fill=red!100] (-1.34,0.6)--(-1,0.3)--(-1.34,0.3)--cycle;
	\node[ right] at (0,1) {$O$};
	\node[right] at (2.2,0.7) {{\footnotesize{$B_{\textbf{e}_1}(x_1)$}}};
	\node[right] at (2.2,0.3) {{\footnotesize{$B_{\textbf{e}_1}(x_1+t)$}}};
	\node[below ] at (0.8,2.3) {{\footnotesize{$B_{\textbf{e}_2}(x_2)$}}};
	\draw (0,1.8) circle (1.8);
	\draw (-2.2,0.6)--(2.2,0.6);
	\draw (-2.2,0.3)--(2.2,0.3);
	\node[ right] at (2.5,1.4) {$\mathbb{U}^{n+1}$};
	\end{tikzpicture}
	\hspace{2.5em}
	\begin{tikzpicture}[scale=0.8][>=Stealth] 
	\draw[->](-2.5,0)--(2.5,0) node[below right]{$\textbf{y}$};
	\draw[->](0,-0.5)--(0,4) node[above]{$y_{n+1}$};
	\foreach \y in {0} \draw (1pt, \y) -- (-1pt, \y) node[anchor=east] {$\y$};
	\filldraw[opacity=0.6, draw=red!70, fill=red!100] (1.658,2)--(1.658,2.5)--(1.789,2)--cycle;
	\filldraw[opacity=0.6, draw=red!70, fill=red!100] (-1.658,2)--(-1.658,2.5)--(-1.789,2)--cycle;
	\fill (0,2.7) circle (1pt);
	\node[ left] at (0,2.7) {$1$};
	\node[ right] at (0,2.7) {$O$};
	\node[right] at (2.2,2.6) {{\footnotesize{$B_{\textbf{e}_1}(x_1)$}}};
	\node[right] at (2.2,2) {{\footnotesize{$B_{\textbf{e}_1}(x_1+t)$}}};
	\node[below ] at (0.8,1.8) {{\footnotesize{$B_{\textbf{e}_2}(x_2)$}}};
	\draw (0,1.8) circle (1.8);
	\draw (-2.2,2.5)--(2.2,2.5);
	\draw (-2.2,2)--(2.2,2);
	\node[ right] at (2.5,1.4) {$\mathbb{U}^{n+1}$};
	\end{tikzpicture}
	\caption{$M(t)\Delta P_t$}
	\label{Fig-4-2}
\end{figure}

\begin{lem}\label{Lem4.6}
	Let $\{\textbf{e}_1,\ldots,\textbf{e}_m\}\subset \mathbb{S}^n\ (m\ge2)$ and $x\in\mathbb{R}_{+}^m$, denote by $P(x)$ the horospherical Wulff shape associated with $(\{\textbf{e}_i\}_{i=1}^m,x)$. Then $V(P(x))$ is a $C^1$ function of $x$ in $\mathbb{R}_{+}^m$, and
	\begin{equation}\label{4.10}
	\frac{\partial}{\partial x_i}V(P(x))=S(P(x),\textbf{e}_i),\qquad 1\le i\le m.
	\end{equation}
\end{lem}

\begin{proof}
    Denote by $\{\delta_k\}_{k=1}^{m}$ the standard basis of $\mathbb{R}^m$. Let us work in the upper half-space model $(\mathbb{U}^{n+1},g_U)$, see Figure \ref{Fig-4-3}.
	
	Let $P_t= \mathscr{P}_{t}(P(x),\mathring{F}(P(x),\textbf{e}_i))$ as defined in Lemma \ref{Lem4.3}, $P_j(x)=\overline{B}_{\textbf{e}_i}(x_i)\cap \overline{B}_{\textbf{e}_j}(x_j)$ and $P_{t,j}=\mathscr{P}_{t}(P_j(x),\mathring{F}(P_j(x),\textbf{e}_i))$. For any $t>0$ and $j=1,\ldots,m$, define
	\begin{align*}
	M_{j}(t)=\left(\overline{B}_{\textbf{e}_i}(x_i+t)\backslash\overline{ B}_{\textbf{e}_i}(x_i)\right)\bigcap\overline{B}_{\textbf{e}_j}(x_j).
	\end{align*}
     Applying Lemma \ref{Lem4.5}, we have
	\begin{align}\label{4.11}
	V(M_j(t)\Delta P_{t,j})\le C_j t^2,\quad \text{as}\ t\to 0^+,\quad j\neq i,
	\end{align}
	where $C_j$ is a positive constant depending only on $\{\textbf{e}_i,\textbf{e}_j\}$ and $(x_i,x_j)$. 
	
	It follows from the definition of $P(x)$ that 
	\begin{align}\label{4.12}
	P(x+t\delta_i)\backslash P(x)
	=\left(\overline{B}_{\textbf{e}_i}(x_i+t)\backslash\overline{ B}_{\textbf{e}_i}(x_i)\right)
	\bigcap\left(\bigcap\limits_{j\neq i}{\overline{B}_{\textbf{e}_j}(x_j)}\right)
	=\bigcap\limits_{j\neq i}{M_j(t)}.
	\end{align}
	By (\ref{4.6}), it is easily verified that $P_t=\bigcap\limits_{j\neq i}{P_{t,j}}$. 
    By (\ref{4.11}) and (\ref{4.12}), we have
	\begin{align*}
	V\left(\left(P(x+t\delta_i)\backslash P(x)\right)\backslash P_t\right)
	=V\left(\bigcup\limits_{j\neq i}((\bigcap\limits_{k\neq i}{M_k(t)})\backslash P_{t,j})\right)
	\le \sum\limits_{j\neq i}V\left(M_j(t)\backslash P_{t,j}\right)
	\le C t^2,\ \text{as}\ t\to 0^+,
	\end{align*}
	and
	\begin{align*}
	V\left(P_t\backslash\left(P(x+t\delta_i)\backslash P(x)\right)\right)
	=V\left(\bigcup\limits_{j\neq i}((\bigcap\limits_{k\neq i}{P_{t,k}})\backslash M_j(t))\right)
	\le \sum\limits_{j\neq i}V\left(P_{t,j}\backslash M_j(t)\right)
	\le C t^2,\ \text{as}\ t\to 0^+,
	\end{align*}
	where $C$ is a positive constant depending only on $\{\textbf{e}_k\}_{k=1}^m$ and $x$. Thus we obtain
	\begin{align}\label{4.13}
	V\left(\left(P(x+t\delta_i)\backslash P(x)\right)\Delta P_t\right)
	\le 2C t^2,\quad \text{as}\ t\to 0^+.
	\end{align}

	On the other hand, if $0>t>-\epsilon$ for sufficiently small $0<\epsilon<\min_i x_i$, then $P(x)=P((x+t\delta_i)+|t|\delta_i)$. Let $$\tilde{P}_{|t|}=\mathscr{P}_{|t|}( P(x+t\delta_i),\mathring{F}(P(x+t\delta_i), \textbf{e}_i) ).$$ 
	It follows from (\ref{4.13}) that
	\begin{align}\label{4.14}
	V\left(\left(P(x)\backslash P(x+t\delta_i)\right)\Delta \tilde{P}_{|t|}\right)
	\le \tilde{C} t^2,\quad \text{as}\ t\to 0^-,
	\end{align}
	where $\tilde{C}$ is a positive constant depending only on $\{\textbf{e}_k\}_{k=1}^m$, $x$ and $\epsilon$. 
	
	Note that $\Phi_n(P(x),F(P(x),\textbf{e}_i))=S(P(x),\textbf{e}_i)$. Using (\ref{4.13}) and Lemma \ref{Lem4.3}, we obtain
	\begin{equation*}
	\begin{split}
	\lim\limits_{t\to 0^+}\frac{V(P(x+t\delta_i))-V(P(x))}{t}
	=\lim\limits_{t\to 0^+}\frac{V(P_t)}{t}
	=S(P(x),\textbf{e}_i).
	\end{split}
	\end{equation*}
	As shown in the proof of Theorem \ref{Thm3.8}, $P(x+t\delta_i)$ converges to $P(x)$ in the Hausdorff metric as $t \to 0^{-}$.
	Using (\ref{4.14}), (\ref{4.7}) and Lemma \ref{Lem4.4}, we obtain
	\begin{equation*}
	\begin{split}
	\lim\limits_{t\to 0^-}\frac{V(P(x+t\delta_i))-V(P(x))}{t}
	=\lim\limits_{t\to 0^-}\frac{V(\tilde{P}_{|t|})}{-t}
	=\lim\limits_{t\to 0^-}\frac{e^{n|t|}-1}{n|t|}S(P(x+t\delta_i),\textbf{e}_i)
	=S(P(x),\textbf{e}_i).
	\end{split}
	\end{equation*}
	Consequently, we have
	\begin{equation*}
	\frac{\partial}{\partial x_i}V(P(x))
	=S(P(x),\textbf{e}_i).
	\end{equation*}
	This completes the proof of Lemma \ref{Lem4.6}.
\end{proof}

\begin{figure}[htbp]
	\centering
	\begin{tikzpicture}[scale=1][>=Stealth] 
	\draw[->](-3.3,0)--(3.3,0) node[below right]{$\textbf{y}$};
	\draw[->](0,-0.5)--(0,4) node[above]{$y_{n+1}$};
	\foreach \y in {0} \draw (1pt, \y) -- (-1pt, \y) node[anchor=east] {$\y$};
	\filldraw[opacity=0.6, draw=blue!70, fill=blue!100] (0.808,1.1)--(0.709,0.9)--(0.575,0.7)--(-0.973,0.7)--(-1.116,0.9)--(-1.223,1.1)--cycle;
	\fill (0,1.5) circle (1pt);
	\node[ left] at (0,1.5) {$1$};
	\node[ right] at (0,1.5) {$O$};
	\node[right] at (2.7,1.2) {{\footnotesize{$B_{\textbf{e}_1}(x_1)$}}};
	\node[right] at (2.7,0.9){{\footnotesize{$B_{\textbf{e}_1}(x_1+t)$}}};
	\node[below ] at (-1.97,2.6){{\footnotesize{$B_{\textbf{e}_2}(x_2)$}}};
	\node[below ] at (1.6,2.6){{\footnotesize{$B_{\textbf{e}_3}(x_3)$}}};
	\node[below ] at (-0.6,2.4) {{\footnotesize{$P(x)$}}};
	\draw (-0.85,1.8) circle (1.8);
	\draw (0.5,1.9) circle (1.9);
	\draw (-3,1.1)--(2.7,1.1);
	\draw (-3,0.7)--(2.7,0.7);
	\node[ right] at (3,2) {$\mathbb{U}^{n+1}$};
	\end{tikzpicture}
	\caption{Blue domain: $P(x+t\delta_1)\backslash P(x)$.}
	\label{Fig-4-3}
\end{figure}

\section{A boundedness lemma for h-convex bodies}\label{sec:5}

The following Lemma \ref{Lem5.1} is devoted to obtaining an upper bound for the horospherical support functions of a family of h-convex bodies with bounded volume, which will be used in the proof of the case $p<0$ of Theorem \ref{Thm1.1} in subsection \ref{sec:6.2} below.

\begin{lem}\label{Lem5.1}
	Let $M$ be a positive number and $K$ be a h-convex body in $\mathcal{K}_o^h$ with $V(K)\le M$. Then there exists a constant $C$ depending only on $M$ such that
	\begin{equation}\label{5.1}
	u(K,\textbf{e})\le C,\quad\forall\ \textbf{e}\in \mathbb{S}^n.
	\end{equation}
\end{lem}

\begin{proof}
	We work in the upper half-space model $(\mathbb{U}^{n+1},g_U)$, and denote $O=(\textbf{0},1)\in\mathbb{U}^{n+1}$. For any $r>0$, let $P=(\textbf{0},e^r)$ be a point in $\mathbb{U}^{n+1}$ such that $d(O,P)=r$. Consider the h-convex hull of $O$ and $P$ defined by (see Figure \ref{fig-Tr})
	\begin{align}\label{5.2}
	T(r)=\bigcap\left\{\overline{B}_{\textbf{e}}(s):\ O,P\in\partial B_{\textbf{e}}(s)\right\}.
	\end{align}
	Here we regard $\textbf{e}$ as a point in $\partial\mathbb{U}^{n+1}=(\mathbb{R}^n\times\{0\})\cup\{\infty\}$.
	
	Now we proceed to prove that $V(T(r)) \to +\infty$ as $r \to \infty$.
	Note that the horosphere $\partial B_{\textbf{e}}(s)$ that passes through $O$ and $P$ satisfies $s=0$ and $\textbf{e}=(e^{\frac{r}{2}}\boldsymbol{\theta},0)$ for some $\boldsymbol{\theta}\in\mathbb{S}^{n-1}$. Then the horo-ball $B_{\textbf{e}}(s)$ is given by
	\begin{align*}
	B_{\textbf{e}}(s)=\left\{(\textbf{y},y_{n+1})\in\mathbb{U}^{n+1}:\ 
	|\textbf{y}-e^{\frac{r}{2}}\boldsymbol{\theta}|^2+\left(y_{n+1}-\frac{e^r+1}{2}\right)^2<\left(\frac{e^r+1}{2}\right)^2\right\}.
	\end{align*}
	Thus $T(r)$ can be represented as
	\begin{align*}
	T(r)=\left\{(\tau\boldsymbol{\theta},\bar{y})\in\mathbb{U}^{n+1}:\
	0\le \tau\le S(r,\bar{y}),\ 1\le \bar{y}\le e^r,\ \boldsymbol{\theta}\in\mathbb{S}^{n-1}\right\},
	\end{align*}
	where $S(r,\bar{y})=\sqrt{(e^r+1)\bar{y}-\bar{y}^2}-e^{\frac{r}{2}}$. A direct computation gives 
	\begin{align*}
	V(T(r)) &=\int_1^{e^r}d\bar{y}\int_{T(r)\cap\{y_{n+1}=\bar{y}\}}\frac{d\textbf{y}}{\bar{y}^{n+1}}
	=\omega_n \int_1^{e^r} \frac{S(r,\bar{y})^n}{\bar{y}^{n+1}}d\bar{y},
	\end{align*}
	where $\omega_n$ is the volume of the unit ball in $\mathbb{R}^n$. One can check that $S(r, \cdot)$ is concave on $[1, \frac{e^r+1}{2}]$ and  $S(r, 1)=0$.
	Then for $1\le \bar{y}\le\frac{e^r+1}{2}$,
	\begin{equation*}
	S(r, \bar{y}) \geq \frac{S(r, \frac{e^r+1}{2}) - S(r,1)}{ \frac{e^r+1}{2} -1} (\bar{y} -1)
	=\frac{e^{\frac{r}{2}}-1}{e^{\frac{r}{2}}+1}(\bar{y}-1).
	\end{equation*}
	Hence
	\begin{align*}
	V(T(r))  &\ge \omega_n \int_1^{\frac{e^r+1}{2}}\left(\frac{e^{\frac{r}{2}}-1}{e^{\frac{r}{2}}+1}\right)^n\frac{(\bar{y}-1)^n}{\bar{y}^{n+1}}d\bar{y}\\
	&=\omega_n \left(\frac{e^{\frac{r}{2}}-1}{e^{\frac{r}{2}}+1}\right)^n \left[\log\frac{e^r+1}{2}+\sum\limits_{k=0}^{n-1}\frac{(-1)^{n-k}\binom{n}{k}}{n-k}\left(1-\left(\frac{e^r+1}{2}\right)^{k-n}\right)\right].
	\end{align*}
	It follows that the volume $V(T(r)) \to +\infty$  as $r \to \infty$. Therefore, there exists a constant $C$ depending only on $M$ such that
	$V(T(C))>M$. 
	
	To finish the proof, we derive a contradiction. 
    Suppose that there exists $\textbf{e}$ satisfying $u(K,\textbf{e})\ge C$. Let $P \in \partial K \cap  H_{\textbf{e}} (u (K, \textbf{e}))$. Since $K\in\mathcal{K}_o^h$, we have 
	\begin{equation}\label{5.3}
	d(P,O)\ge u(K,\textbf{e})\ge C.
	\end{equation}
	It follows from the h-convexity of $K$ that $K$ contains the h-convex hull of $O$ and $P$, thus
	\begin{equation}\label{5.4}
	V(K)\ge V(T(d(P,O)))>M,
	\end{equation}
	which contradicts the assumption $V(K)\le M$. This completes the proof of Lemma \ref{Lem5.1}.
\end{proof}

\begin{figure}[htbp]
	\centering
	\begin{tikzpicture}[scale=0.6][>=Stealth] 
	\draw[->](-6,0)--(6,0) node[right]{$\textbf{y}$};
	\draw[->](0,-1)--(0,5) node[above]{$y_{n+1}$};
	%\foreach \y in {0, 1} \draw (1pt, \y) -- (-1pt, \y) node[anchor=east] {$\y$};
	\node[ right] at (0,1) {$O$};
	\fill (0,3) circle (2pt);
	\node[ right] at (0,3) {$P$};
	\node[ left] at (-0.2,2) {${T(r)}$};
	\draw (1.732,2) circle (2);
	\draw (-1.732,2) circle (2);
	\filldraw[opacity=0.8, fill=blue!100] (0,3)arc(150:210:2);
	\filldraw[opacity=0.8, fill=blue!100] (0,1)arc(-30:30:2);
	\node[ right] at (4,2) {$\mathbb{U}^{n+1}$};
	\end{tikzpicture}
	\caption{$T(r)$}
	\label{fig-Tr}
\end{figure}

\section{Proof of Theorem \ref{Thm1.1}}\label{sec:6}

Let $a_1,\ldots,a_{N}$ be positive numbers ($N\ge2$), and let $\textbf{e}_1,\ldots,\textbf{e}_{N}$ be unit vectors on $\mathbb{S}^n$. For convenience, we define
$$\mathbb{R}_{*}^{m}=\overline{\mathbb{R}_{+}^{m}}=\{(x_1,\ldots,x_{m})\in\mathbb{R}^{m}:\ x_i\ge0,\ 1\le i\le m\}.$$
From now on, we always denote $x_{m+i}=x_{i}$ for $x\in\mathbb{R}_*^m$ and $1\le i\le m$.

Recall that a discrete measure $\mu=\sum_{i=1}^{N}a_i\delta_{\textbf{e}_i}$ defined on $\mathbb{S}^n$ is called even if $\mu(\omega)=\mu(-\omega)$ for any Borel set $\omega\subset\mathbb{S}^n$, which is equivalent to (up to a permutation)
$$N=2m,\ \textbf{e}_i=-\textbf{e}_{m+i},\ a_i=a_{m+i},\quad \forall \  1\le i\le m.$$

Given any $x\in\mathbb{R}_*^{m}$ and unit vectors $\{\textbf{e}_i\}_{i=1}^{2m}$ satisfying $\textbf{e}_i=-\textbf{e}_{m+i}$ for $1\le i\le m$, we can construct an origin-symmetric h-convex polytope $P(x)$ in $\mathcal{P}^h(\textbf{e}_1,\ldots,\textbf{e}_{2m})$ by
\begin{equation}\label{6.1}
P(x) 
:=P(\{\textbf{e}_i\}_{i=1}^{2m},(x,x))
=\bigcap\limits_{i=1}^{2m}\overline{B}_{\textbf{e}_i}(x_i)
=\bigcap\limits_{i=1}^{m}\left(\overline{B}_{\textbf{e}_i}(x_i)\cap\overline{B}_{\textbf{e}_{m+i}}(x_{i})\right).
\end{equation}
Note that when $x_i=0$ for some $1\le i\le m$, $P(x)$ is exactly the origin point $O$. We define a function $\Phi_p: \mathbb{R}_*^{m} \to \mathbb{R}$ as follows:
\begin{equation}\label{6.2}
\Phi_p(x)=\left\{\begin{array}{l}
\frac{1}{p}\sum\limits_{i=1}^{2m}a_i (e^{p x_i}-1),\quad \text{if}\ p\neq0,\\
\sum\limits_{i=1}^{2m}a_ix_i,\quad\quad\quad\quad\quad   \text{if}\ p=0.
\end{array}\right.
\end{equation}

\subsection{The proof of Theorem \ref{Thm1.1} for $p\geq 0$}\label{sec:6.1}$\ $

In this subsection, we study an optimization problem with natural constraints. Its solution solves the discrete horospherical $p$-Minkowski problem for $p\ge0$ in the even case.

Let $p\ge0$ and $\mu=\sum_{i=1}^{2m}a_i\delta_{\textbf{e}_i}$ be an even discrete measure on $\mathbb{S}^n$. We consider the following optimization problem
\begin{align}\label{6.3}
\sup\left\{V(P(x)):\ x\in\mathbb{R}_*^{m}\ \text{and}\ \Phi_p(x)=1\right\}.
\end{align}

Since $M=\{x\in\mathbb{R}_*^{m}:\ \Phi_p(x)=1\}$ is compact, and the function $V(P(x))$ is continuous on $M$, there exists a point $z\in M$ such that
\begin{equation}\label{6.4}
	V(P(z))=\max\limits_{x\in M} V(P(x)).
\end{equation}
	
It follows from $V(P(z))>0$ and the origin-symmetry of $P(z)$ that $z$ lies the interior of $\mathbb{R}_*^{m}$. Using the Lagrange multiplier method and the variational formula in Lemma \ref{Lem4.6}, there exists a constant $\lambda$ such that
\begin{equation}\label{6.5}
\begin{split}
	0 &=\left.\frac{\partial}{\partial x_i}\right|_{x=z}\left[V(P(x))-\lambda\left(\Phi_p(x)-1\right)\right]\\
	&=S({P(z)},\textbf{e}_i)+S({P(z)},\textbf{e}_{m+i})-\lambda a_i e^{p z_i}-\lambda a_{m+i} e^{p z_{m+i}},\quad i=1,\ldots,m. 
\end{split}
\end{equation}
Using $a_{i}=a_{m+i},\ z_{i}=z_{m+i}$ and $S({P(z)},\textbf{e}_i)=S({P(z)},\textbf{e}_{m+i})$, we obtain
\begin{equation}\label{6.6}
	S({P(z)},\textbf{e}_i)=\lambda a_i e^{pz_i},\quad i=1,\ldots,m,
\end{equation}
and
\begin{equation}\label{6.7}
	\lambda=\frac{\sum S({P(z)},\textbf{e}_i)}{\sum a_i e^{p z_i}}=\frac{S(P(z))}{\sum a_i e^{p z_i}}>0.
\end{equation}
We can conclude that $S({P(z)},\textbf{e}_i)>0$ and then $u(P(z),\textbf{e}_i)=z_i$. Consequently,
\begin{equation}\label{6.8}
\begin{split}
	\mu &=\sum\limits_{i=1}^{2m}a_i\delta_{\textbf{e}_i}(\cdot)\\
	&=\frac{1}{\lambda}\sum\limits_{i=1}^{2m}e^{-p u(P(z),\textbf{e}_i)}S({P(z)},\textbf{e}_i)\delta_{\textbf{e}_i}(\cdot)\\
	&=\frac{1}{\lambda} S_p(P(z),\cdot).
\end{split}
\end{equation}
This completes the proof of Theorem \ref{Thm1.1} for $p\ge 0$.

\begin{rem}\label{Rem6.1}
	In particular, the case $p=0$ of Theorem \ref{Thm1.1} corresponds to the prescribed discrete horospherical surface area measure problem. Namely, there exists an origin-symmetric h-convex polytope $P$ that has precisely $\textbf{e}_1,\ldots,\textbf{e}_{2m}$ as its horospherical normals, and $a_1,\ldots,a_{2m}$ are multiples of corresponding areas of the facets of $P$.
\end{rem}

\begin{rem}\label{Rem6.2}
	For the case $p<0$, the corresponding optimization problem might be 
	\begin{equation}\label{6.9}
    \sup\left\{V(P(x)):\ x\in\mathbb{R}_*^{m}\ \text{and}\ \Phi_p(x)=1\right\}
	\end{equation}
    or
    \begin{equation}\label{6.10}
	\inf\left\{V(P(x)):\ x\in\mathbb{R}_*^{m}\ \text{and}\ \Phi_p(x)=1\right\}.
	\end{equation}
	However, these problems are not suitable for obtaining solutions to the discrete horospherical $p$-Minkowski problem for $p<0$ in the even case for the following reasons:
	\begin{enumerate}[(i)]
		\item Since the set $M=\{x\in\mathbb{R}_*^{m}:\ \Phi_p(x)=1\}$ is not bounded when $p<0$, the elements in $M$ may not attain the supremum in the problem (\ref{6.9}).
		
		\item For the infimum in the problem (\ref{6.10}), we can choose some values $\{a_1,\ldots,a_{2m}\}$ such that there exists $z\in M$ with $z_1=0$, then  $\inf_{x\in M}V(P(x))=V(P(z))=0$. In this case, $z$ is on the boundary of $\mathbb{R}_*^{m}$ and the variational argument fails.
	\end{enumerate}
\end{rem}

\subsection{The proof of Theorem \ref{Thm1.1} for $p<0$}\label{sec:6.2}$\ $

In this subsection, we study the discrete horospherical $p$-Minkowski problem for $p<0$ in the even case.

\begin{defn}
	A vector $u=(u_1,\ldots,u_{2m})\in\mathbb{R}_*^{2m}$ is called \textit{admissible} if the horospherical Wulff shape associated with $(\{\textbf{e}_i\}_{i=1}^{2m},u)$ satisfies
	\begin{align}\label{6.11}
	u(P(\{\textbf{e}_i\}_{i=1}^{2m},u),\textbf{e}_i)=u_i,\quad 1\le i\le 2m.
	\end{align}
\end{defn}

We need the following lemma in the study of the optimization problem.
% infimum point of $\Phi_p (x)$.
\begin{lem}\label{Lem6.4}
	Let $p<0$ and $x\in\mathbb{R}_*^{m}$ satisfying $V(P(x))=V_0$. Then there exists an admissible $\tilde{x}\in\mathbb{R}_*^{m}$ such that $V(P(\tilde{x}))=V_0$ and $\Phi_p(\tilde{x})\le\Phi_p(x)$.
\end{lem}
\begin{proof}
	Let $x\in\mathbb{R}_*^{m}$, and denote
	\begin{equation*}
	\Tilde{x}_i=u(P(x),\textbf{e}_i),\quad\Tilde{x}=(\tilde{x}_i)_{1\le i\le m}\in \mathbb{R}_*^{m}.
	\end{equation*}
	Using (\ref{4.3}), we obtain
	\begin{align}\label{6.12}
	\Tilde{x}_i=u(P(x),\textbf{e}_i)\le x_i,\quad 1\le i\le m.
	\end{align}
	It follows that
	\begin{equation*}
	\Phi_p(\Tilde{x})
	=\frac{1}{p}\sum\limits_{i=1}^{2m}a_i (e^{p\tilde{x}_i}-1)
	\le \frac{1}{p}\sum\limits_{i=1}^{2m}a_i (e^{p x_i}-1)
	=\Phi_p(x).
	\end{equation*}
	
	By (\ref{2.3}), we have $P(x)\subset \overline{B}_{\textbf{e}_i}(\tilde{x}_i)$ for all $1\le i\le m$. Then
	\begin{align}\label{6.13}
	P(x)\subset \bigcap_{i=1}^{2m} \overline{B}_{\textbf{e}_i}(\tilde{x}_i)=P(\tilde{x}).
	\end{align}
	On the other hand, it follows from (\ref{6.12}) that $P(\tilde{x})\subset P(x)$. Combining this with (\ref{6.13}), we have
	\begin{align*}
	P(\tilde{x})=P(x),\quad  
	V(P(\tilde{x}))=V_0.
	\end{align*}
	It is clear that
	\begin{align*}
	u(P(\tilde{x}),\textbf{e}_i)=\tilde{x}_i,\quad  1\le i\le m.
	\end{align*}
	Hence $\tilde{x}$ is admissible. This completes the proof of Lemma \ref{Lem6.4}. 
\end{proof}

Let $p<0$ and $\mu=\sum_{i=1}^{2m}a_i\delta_{\textbf{e}_i}$ be an even discrete measure on $\mathbb{S}^n$. Let us consider the following optimization problem
\begin{align}\label{6.14}
\inf\left\{\Phi_p(x):\ x\in\mathbb{R}_*^{m}\ \text{and}\ V(P(x))=V_0 \right\},
\end{align}
where $V_0$ is a given positive constant. 
% \begin{equation}\label{6.13}
% \Phi_p(x)=
% \frac{1}{p}\sum\limits_{i=1}^{2m}a_i (e^{px_i}-1).
% \end{equation}

%Assume that $p<0$. 
Note that the function $\Phi_p(x)$ defined in (\ref{6.2}) is non-negative. Let us take a minimizing sequence $\{x^{(k)}\}\subset\mathbb{R}_*^{m}$ for the optimization problem (\ref{6.14}), i.e.
\begin{equation}\label{6.15}
    V(P(x^{(k)}))=V_0,
\end{equation}
and
\begin{equation*}
    \lim\limits_{k\to\infty}\Phi_p(x^{(k)})=\inf\left\{\Phi_p(x):\ x\in\mathbb{R}_*^{m}\ \text{and}\  V(P(x))=V_0\right\}.
\end{equation*}
By Lemma \ref{Lem6.4}, we may assume that each vector in the sequence $\{x^{(k)}\}$ is admissible. Combining this with Lemma \ref{Lem5.1}, we conclude that $\{x^{(k)}\}$ is bounded. On the other hand, it follows from (\ref{6.15}) and the origin-symmetry of $P(x^{(k)})$ that $\{x_i^{(k)}\}$ has a positive uniform lower bound for all $i,k$.

Therefore, by use of Theorem \ref{Thm3.8} and Corollary \ref{Cor3.6}, there exists a subsequence of $P(x^{(k)})$ (still denoted by $P(x^{(k)})$) that converges to an origin-symmetric h-convex polytope $P_0\in\mathcal{P}^h(\textbf{e}_1,\ldots,\textbf{e}_{2m})$ satisfying $V(P_0)=V_0$.
Let $x^{(0)}=(u(P_0,\textbf{e}_i))_{1\le i\le m}\in\mathbb{R}_*^{m}$. Then $P(x^{(0)})=P_0$ and $V(P(x^{(0)}))=V_0$. Since $P(x^{(k)})$ converges to $P_0$ in Hausdorff metric, we get
\begin{equation}\label{6.16}
    x^{(k)}\to x^{(0)}\quad \text{and}\quad \Phi_p(x^{(0)})=\inf\left\{\Phi_p(x):\ V(P(x))=V_0\right\}.
\end{equation}
Due to the origin-symmetry of $P(x^{(0)})$ and $V(P(x^{(0)}))=V_0>0$, we have that $x^{(0)}$ is in the interior of $\mathbb{R}_*^{m}$. By the Lagrange multiplier method and the variational formula in Lemma \ref{Lem4.6}, there exists a constant $\lambda$ such that
\begin{equation*}
\begin{split}
    0 &=\left.\frac{\partial}{\partial x_i}\right|_{x=x^{(0)}}\left[\Phi_p(x)-\lambda (V(P(x))-V_0) \right]\\
    &=a_i e^{p x_i^{(0)}}+a_{m+i} e^{p x_{m+i}^{(0)}}-\lambda S({P(x^{(0)})},\textbf{e}_i)-\lambda S({P(x^{(0)})},\textbf{e}_{m+i}),\quad i=1,\ldots,m.
\end{split}
\end{equation*}
Using $a_i=a_{m+i},\ x_i^{(0)}=x_{m+i}^{(0)}$ and $S({P(x^{(0)})},\textbf{e}_i)=S({P(x^{(0)})},\textbf{e}_{m+i})$, we have
\begin{equation}\label{6.17}
    a_i e^{p x_i^{(0)}}=\lambda S({P(x^{(0)})},\textbf{e}_i),\quad i=1,\ldots,m,
\end{equation}
and
\begin{equation*}
    \lambda=\frac{\sum a_i e^{p x_i^{(0)}}}{S(P(x^{(0)}))}>0.
\end{equation*}
Therefore, $S({P(x^{(0)})},\textbf{e}_i)>0$ and then $u(P(x^{(0)}),\textbf{e}_i)=x_i^{(0)}$. This together with (\ref{6.17}) implies
\begin{equation*}
\begin{aligned}
    \mu =&\sum\limits_{i=1}^{2m}a_i\delta_{\textbf{e}_i}(\cdot)\\
    =&\lambda\sum\limits_{i=1}^{2m}e^{-p u(P(x^{(0)}),\textbf{e}_i)}S(P(x^{(0)}),\textbf{e}_i)\delta_{\textbf{e}_i}(\cdot)\\
    =&\lambda S_p(P(x^{(0)}),\cdot).
\end{aligned}
\end{equation*}
Then $P(x^{(0)})$ is the desired solution. We complete the proof of Theorem \ref{Thm1.1} for $p<0$.

\subsection{Further discussion}$\ $

In Theorem \ref{Thm1.1}, we assumed that the prescribed measure $\mu$ is even. However, if $\mu$ is not even, then there may exist necessary conditions on $\mu$ for the existence of solutions to Problem \ref{Problem 1.2}. For the smooth case, one can refer to the Kazdan-Warner type obstruction of the horosphercial Minkowski problem in \cite{LX}. For the discrete case, we provide the following example.

\begin{ex}[Problem \ref{Problem 1.1} when $N=2$]
For any unit vectors $\textbf{e}_1, \textbf{e}_2\in\mathbb{S}^n$ and positive numbers $x_1, x_2$, denote by $P$ the horospherical Wulff shape associated with $(\{\textbf{e}_1,\textbf{e}_2\},(x_1,x_2))$, i.e., $P=\bar{B}_{\textbf{e}_1}(x_1)\cap \bar{B}_{\textbf{e}_2}(x_2)$. A direct calculation shows that
\begin{align}\label{6.18}
    S(P,\cdot)=\omega_n\left(e^{x_1+x_2}-1\right)^{\frac{n}{2}}\left(\delta_{\textbf{e}_1}(\cdot)+\delta_{\textbf{e}_2}(\cdot)\right).
\end{align}

Therefore, given a discrete measure $\mu=a_1\delta_{\textbf{e}_1}+a_2\delta_{\textbf{e}_2}$, there exists a solution to Problem \ref{Problem 1.1} for $\mu$ if and only if $a_1=a_2$. Moreover, when $a_1=a_2$, it follows from (\ref{6.18}) that this problem has infinitely many solutions.
\end{ex}

Finally, we restate Problem \ref{Problem 1.2} in the case that $\mu$ is not even. This problem is of great interest in the discrete horospherical $p$-Brunn-Minkowski theory.
\begin{prob}\label{Problem 6.1}
Suppose that a discrete measure $\mu$ on $\mathbb{S}^n$ is not even. Find necessary and sufficient conditions on $\mu$ so that there exists a h-convex polytope $P$ in $\mathbb{H}^{n+1}$ whose horospherical $p$-surface area measure $S_p(P,\cdot)$ is a multiple of the given measure $\mu$.
\end{prob}

\end{document}